\documentclass[a4paper,12pt,reqno]{amsart}
\usepackage{amsfonts}
\usepackage{amsmath}
\usepackage{amssymb}
\usepackage[a4paper]{geometry}
\usepackage{mathrsfs}
\usepackage{csquotes}
\usepackage{mathrsfs}
\usepackage[toc,page]{appendix}
\usepackage[colorlinks]{hyperref}
\renewcommand\eqref[1]{(\ref{#1})} %Need with hyperref
\numberwithin{equation}{section}
\theoremstyle{plain}
\newtheorem{thm}{Theorem}[section]
\newtheorem{prop}[thm]{Proposition}

 \newtheorem{exa}[thm]{Example}
\theoremstyle{definition}
\newtheorem{defn}[thm]{Definition}
\newtheorem{rem}[thm]{Remark}

\newcommand{\G}{\mathbb G}

\def\e[#1]{{\textrm{e}}^{#1}}

\def\G{{\mathbb G}}

\begin{document}

   \title[Quantizations on the Engel and the Cartan groups]
 %  {Critical cases of Sobolev inequalities on graded groups and applications to fractional order nonlinear subelliptic equations}
 {Quantizations on the Engel and the Cartan groups}

\author[M. Chatzakou]{Marianna Chatzakou}
\address{
  Marianna Chatzakou:
  \endgraf
  Department of Mathematics
  \endgraf
  Imperial College London
  \endgraf
  180 Queen's Gate, London SW7 2AZ
  \endgraf
  United Kingdom
  \endgraf
  {\it E-mail address} {\rm m.chatzakou16@imperial.ac.uk}
  }


     \begin{abstract}
     This work aims to develop a global quantization in the concrete settings of two graded nilpotent Lie groups of 3-step; namely of the Engel group and the Cartan group. We provide a preliminary analysis on the structure and the representations of the aforementioned groups, and their corresponding Lie algebras. In addition, the explicit formulas for the difference operators in the two settings are derived, constituting the necessary prerequisites for the constructions of the $\Psi^{m}_{\rho,\delta}$ classes of symbols in both cases. In the case of the Engel group, the relation between the Kohn-Nirenberg quantization and the representations of the Engel group enables us to express operators in this setting in terms of quantization of symbols in the Euclidean space. We illustrate the $L^p-L^q$ boundedness for spectral multipliers theorems with  particular examples of operators in both groups, that yield appropriate Sobolev-type inequalities. As a further application of the above, we get some consequences for the $L^p-L^q$ norm for the heat kernel of those particular examples of operators in both settings. 
     \end{abstract}
     \maketitle

       \tableofcontents

\section{Introduction}
\label{SEC:intro}
Since their initiation in the sixties, pseudo-differential operators  have become a standard tool in the study of partial differential equations. In the Euclidean setting pseudo-differential operators are defined \textit{globally} as the \textit{quantization} of a smooth function called the \textit{symbol} via the Euclidean Fourier transform. In this setting, pseudo-differential operators are linear operators usually defined by
\begin{equation}\label{khon-euc}
a(x,D)u(x)=\int_{\mathbb{R}^n}e^{2\pi i x \cdot\xi}a(x,\xi) \hat{u}(\xi)\,d\xi\,,\quad u \in \mathcal{S}(\mathbb{R}^n)\,,
\end{equation} 
where $\hat{u}$ denotes the Euclidean Fourier transform of $u$ in the Schwartz space $\mathcal{S}(\mathbb{R}^n)$.  In \eqref{khon-euc} the pseudo-differential operator $a(x,D)$ arises as the \textit{Kohn-Nirenberg quantization} of the symbol $a$. \\
\indent The utility of pseudo-differential operators is due to the fact that they carry similar properties as their differential counterparts, such as the behaviour of a generalised notion of the order under composition and adjunction, as well as bounded properties between suitable Sobolev spaces. Adjunction and composition formulas are expressed in term of the corresponding symbols, and thus, the \textit{symbolic calculus} is exactly the calculus on the symbolic side that gives rise to the calculus of operators, while also the calculus  contains the parametrices of the elliptic pseudo-differential operators. \\
\indent It is natural question whether a global pseudo-differential calculus can be defined in other, more complicated, settings. To this end, in 2010 a global symbolic calculus was developed on compact Lie groups in \cite{RT10}. In that work the authors defined the \textit{global symbol} of a pseudo-differential operator via the group Fourier transform of the right-convolution kernel (on any connected manifold pseudo-differential operators can always be defined locally via local charts). They provide a definition of the classes of symbols in this setting, so that the associated  (via a suitable quantization procedure) operators form an algebra of operators  ‘close enough’  to the one in the Euclidean setting. The most pivotal part of this work was the definition of \textit{difference operators}, which generalises the derivatives in the Fourier variable from the Euclidean case, and thus allows expressing the pseudo-differential behaviour precisely on the group.  \\
\indent Later on, the case of graded nilpotent Lie groups was treated in \cite{FR16}. However, under this new consideration the results developed in the compact case cannot be readily extended. In particular, they dealt with the technical difficulties that arise from the fact that the dual of the group is no longer discrete and the unitary irreducible representations are infinite dimensional, and more crucially, from the fact that the Laplace-Betrami operator has to be replaced by operators associated with the group via its Lie algebra structure. On stratified Lie groups these are the sub-Laplacians, or more generally, the so-called Rockland operator, and such operators are no longer elliptic but hypoelliptic.\\
 \indent In their monograph, the authors study the global quantization of operators on graded Lie groups, aiming to provide an intrinsic symbolic calculus of the operators in this setting. They give an adequate definition of the difference operators in this setting, which is somehow closer to the one in the Euclidean setting, and therefore, by proving symbolic estimates on the functional calculus for the sub-Laplacian, or the Rockland operator, they provide a way to obtain the formulas for composition and adjunction.  The origins of this go back to the 1970's with the works of E. Stein, G. Folland and L. Rothschild (see e.g. \cite{FS74},\cite{RS77}) among many others, motivated by the study of differential operators on CR or compact manifolds.\\
  \indent  In terms of the developed calculus there, to the author knowledge, most of the versions of the global calculi of operators on homogeneous Lie groups (nilpotent Lie groups endowed with some structure of dilations) that have appeared are, except for a few notable exceptions, calculi of left-invariant operators. Since the seventies, the works on a non-invariant pseudo-differential calculus were limited to \cite{Dyn76}  where S. Dynin considers certain operators on the Heisenberg groups, and to \cite{Fol94}, where G. Folland developed a calculus on any homogeneous group where the classes are given in terms of the kernel, i.e., they are not symbolic. Another version of a non-invariant calculus on any homogeneous group is given in \cite{CGGP92} but this is not symbolic as well. Finally, in \cite{Tayl84} M. Taylor describes a way that one can obtain a symbolic (non-invariant) calculus by defining a \textit{general quantization} and the \textit{general symbols} on any unimodular type I group, but chose to restrict his analysis mainly to invariant operators on the Heisenberg groups.\\
\indent The basic example, apart from $\mathbb{R}^n$, of a nilpotent Lie group, is the Heisenberg group where the associated Fourier analysis is very well studied. In  particular, most of the works that concern the non-invariant symbolic calculus on nilpotent Lie groups are restricted to the Heisenberg groups (or to manifolds having the Heisenberg group as local model), and this is mainly due to the link between the representations of the Heisenberg groups and the Weyl quantization that enables the development of the pseudo-differential calculus on these groups with scalar-valued symbols that depend on parameters. To the best of our knowledge, the only non-invariant calculi with scalar-valued symbols on the Heisenberg groups has been developed in \cite{BFKG12} by H. Bahouri, C. Fermanian-Kammerer and I.Gallagher, while in \cite[Chapter 6]{FR16} one can find  the application of the general theory on graded Lie groups to this particular setting- partially these results had previously been announced in \cite{FR14}. The last two works differ in the conditions on the symbol classes for  small values of the parameters.\\
\indent Besides the big amount of work devoted to the case of the Heisenberg group which is of 2-step, the same motivating aspects appear as well for other graded Lie groups. In particular, our interest in the two graded nilpotent Lie groups of 3-step considered in this work; namely the Engel and the Cartan group (or the generalised Dido problem), is justified by the fact that they are perhaps the most elementary nilpotent Lie groups of step higher that two not yet studied in the context of global pseudo-differential calculus.\\
 \indent  To lay down the necessary foundation for the development of the symbol classes on a nilpotent Lie group, one needs first to study the positive Rockland operators on the group and the associated Sobolev spaces, later on to study and generalise the group Fourier transform and finally to find the concrete formulas for the difference operators on the group. \\
\indent  The paper is organised as follows. In Section \ref{SEC:prelboth}, we give the necessary preliminaries for the construction of the classes of symbols as explained above. In Section \ref{SEC:Groups} we explain the precise settings of our investigation for the groups we consider, including a description of their dual, and of the group Fourier transform in each setting, which, for the case of the Engel group, leads to symbols parametrized by $(\lambda,\mu)$-the co-adjoint orbits. On Section \ref{SEC:Dif}, we find the explicit formulas for the difference operators in both settings.
   The rest of the paper is to some extent independent. In particular, in Section \ref{SEC:multipliers} we prove the $L^p-L^q$ boundedness of spectral multipliers $\phi(D)$ for $D$ being a non-Rockland operator in the setting of the aforementioned groups, as an application of the groups' analogue of the H\"{o}rmander $L^p-L^q$ multiplier theorem  \cite[p.106, Theorem 1.11]{Hor60}. As a straightforward application, we then recover Sobolev-type embedding inequalities. Similar results have been shown at the beginning in the compact case for the case where $D$ is the sub-Laplacian in \cite{HK16}, and in the graded case for $D$ being a Rockland operator, first for the Heisenberg group $\mathbb{H}^n$ in \cite[Section 9.3]{AR17}, and later on, for any stratified group in \cite{RR18}. Additionally, we have shown how the above spectral multiplier estimates can be used to relate spectral properties of particular operators considered here, with time decay rates for propagators for the case of the corresponding heat equations. Similar type techniques, have been used only for the case of the compact Lie group, and one can refer to \cite{AR20}. 
   \section{Preliminaries}
   \label{SEC:prelboth}
   \subsection{The unitary dual and the group Fourier transform}
    We denote by $\hat{\G}$ the unitary dual of the group $\G$, that is, the set of all equivalence classes of irreducible, strongly continuous and unitary representations of $\G$. In our settings, the \textit{Fourier transform} at $\pi \in \hat{\G}$ is defined on $L^{1}(\G,dx)$ by
  \begin{equation}\label{group,fourier,gen}
\pi(\kappa)= \hat{\kappa}(\pi):= \int_{\G}\kappa(x) \pi(x)^{*} dx\,,
\end{equation}
where $dx$ denotes the \textit{Haar measure}, and $\pi^{*}(x)$ the adjoint operator of $\pi(x)$. This defines a linear mapping on the representation space $\mathcal{H}_{\pi}$, i.e., $\hat{\kappa}(\pi):\mathcal{H}_{\pi} \rightarrow\mathcal{H}_{\pi}$.\\
   The group Fourier transform of a vector in the Lie algebra of the group, say $X \in \mathfrak{g}$, at $\pi \in \hat{\G}$ is the operator $\pi(X)$ on $\mathcal{H}_{\pi}^{\infty}\subset \mathcal{H}_{\pi}$, the subspace of smooth vectors, given via
   \begin{equation}\label{defn.infin}
   \pi(X)v=\partial_{t=0}(\pi(exp_{\G}(tX))v\,,
   \end{equation}
   where $exp_{\G}:\mathfrak{g}\rightarrow \G$ is the exponential mapping that identifies $\mathfrak{g}$ with $\G$. By setting $\pi(X^{\alpha})=\pi(X)^{\alpha}$ we extend the group Fourier transform to the Lie algebra $\mathfrak{U}(\mathfrak{g})$.
   \subsection{Rockland operator and Sobolev spaces}
   Let us recall that a (left) \textit{Rockland operator}, say $\mathcal{R}$, on a homogeneous Lie group $\G$\footnote{A graded Lie group is naturally equipped with dilations, i.e., it is homogeneous.} is a positive  left-invariant operator on $\G$ that is homogeneous of degree $\nu$, when for every non-trivial $\pi \in \hat{\G}$ the group Fourier transform $\pi(\mathcal{R})$ is injective on $\mathcal{H}_{\pi}^{\infty}$. Being positive means $(\mathcal{R}f,f)_{L^2(\G)}\geq 0$, for every $f$ in the Schwartz space $\mathcal{S}(\G)$, where as usual 
   \[
   (f_1,f_2)_{L^2(\G)}=\int_{\G}f_1(x)\overline{f_2}(x)\,dx\,.
   \]
   In the stratified case, we can choose $\mathcal{R}=-\mathcal{L}$, where $\mathcal{L}$ is the sub-Laplacian $\sum_{i=1}^{n_1}X_{i}^{2}$, where $\{X_i\}_{i=1}^{n_1}$ are the elements of the first stratum of $\mathfrak{g}$. In this case the homogeneous dimension of $\mathcal{R}$ is $\nu=2$. Furthermore, by \cite[Chapter 4.B]{FS82} any positive Rockland operator $\mathcal{R}$, as an operator on $\mathcal{D}(\G)$, admits a self-adjoint extension on $L^2(\G)$.\\
    Recall also that, for $1 \leq p \leq \infty$, let $\mathcal{R}_p$ denote the extension of $\mathcal{R}$ to $L^{p}(\G)$; for a characterisation of the operator $\mathcal{R}_{p}$, see \cite[Section 4.3.1]{FR14}.\\
   Now, for a fixed Rockland operator $\mathcal{R}$ of homogeneous degree $\nu$, following \cite[Sections 4.4 and 5.1]{FR14}, we define the subsequent spaces, which, for our scope, will be particularly used to give a meaning to the difference operator discussed later in the Section.
   \begin{defn}
   If $p \in [1,\infty)$ and $s \in \mathbb{R}$, we denote by $L^{p}_{s}(\G)$ the Sobolev space obtained by the completion of $\mathcal{S}(\G)$ with respect to the norm 
\[
\|f\|_{L^{p}_{s}(\G)}:=\|(I+\mathcal{R}_{p})^{\frac{s}{\nu}}f\|_{L^{p}(\G)}\,,\quad f \in \mathcal{S}(\G)\,.
\]
   \end{defn}
  Notice that, if $s=0$, then $L_{0}^{p}(\G)=L^{p}(\G)$ for $p \in [1,\infty)$, with $\|\cdot\|_{L_{0}^{p}(\G)}=\|\cdot\|_{L^p(\G)}$.\\
  
  In our setting, the group Fourier transform is an isomorphism between Banach spaces acting from $L^{2}_{a}(\G)$ onto the \textit{$\hat{\G}$-fields of operators} denoted by $L^{2}_{a}(\G)$ defined below:
  \begin{defn}
 Let $a \in \mathbb{R}$. We denote by $L^{2}_{a}(\G)$ the space of fields of operators $\sigma=\{\sigma_{\pi}:\mathcal{H}_{\pi}^{\infty} \rightarrow \mathcal{H}_{\pi}^{a}, \pi \in \hat{\G}\}$ such that 
 \[
 \{\pi(I+\mathcal{R})^{\frac{a}{\nu}}\sigma_{\pi}:\mathcal{H}_{\pi}^{\infty} \rightarrow \mathcal{H}_{\pi}\,, \pi \in \hat{\G}\} \in L^2(\hat{\G})\,.
 \]
 For such a $\sigma$, we set $\|\sigma\|_{L^{2}_{a}(\hat{\G})}:=\|\pi(I+\mathcal{R})^{\frac{a}{\nu}}\sigma_{\pi}\|_{L^2(\hat{\G})}$.\\
 Notice that each $\pi \in \hat{\G}$ is viewed as subset of $\text{Rep}\,\G$.
   \end{defn}
   Recall that the space $\mathcal{H}_{\pi}^{a}$ is the Sobolev space obtained by completion of $\mathcal{H}_{\pi}^{\infty}$ with respect to the norm 
   \[
   \|f\|_{\mathcal{H}_{\pi}^{\infty}}:=\|\pi(I+\mathcal{R})^{\frac{a}{\nu}}f\|_{\mathcal{H}_{\pi}}\,.
   \]
  The next definition is devoted to the the extension of the group Fourier transform to a suitable subset of tempered distribution as we will see later in the section. In particular the following space of fields of operators will turn out to be the image of the group Fourier transform of the later. 
   \begin{defn}\label{Labinf}
Let $\alpha, \beta \in \mathbb{R}$. We denote by $L^{\infty}_{\alpha, \beta}(\hat{\G})$ the space of fields of operators $\sigma = \{ \sigma_{\pi} : \mathcal{H}_{\pi}^{\infty} \rightarrow \mathcal{H}_{\pi}^{b}\,, \pi \in \hat{\G} \}$ such that for some $C>0$
\begin{equation}\label{Linfab}
||\sigma \hat{\phi} ||_{L^{2}_{b}(\hat{\G})} \leq C ||\phi||_{L^{2}_{\alpha}(G)}\,,\quad \phi \in \mathcal{S}(\G)\,,
\end{equation}
and for such a field $\sigma$, the associated norm $\|\sigma\|_{L^{\infty}_{\alpha,\beta}(\hat{\G})}$ is given by the infimum over the constants $C>0$ such that \eqref{Linfab} holds true.
\end{defn}
\begin{rem}\label{eqiv.norms}
We notice that any two norms of the above spaces corresponding to different choices of Rockland operators are equivalent, see \cite[Theorem 4.4.20, Proposition 5.1.7 and Lemma 5.1.7]{FR16}.
\end{rem}
\subsection{Quantization and symbol classes}
As recalled in the introduction there is a natural quantization introduced by \cite{Tayl84} which is valid on any unimodular type-I group due to the Plancherel formula and produces operators $\mathcal{D}(\G) \rightarrow \mathcal{D}^{'}(\G)$, where by $\mathcal{D}(\G)$ we denote the space of smooth and compactly supported functions on $\G$. In particular, the \textit{quantization}, i.e., the mapping $\sigma \mapsto Op(\sigma)$ is analogous to the Kohn-Nirenberg quantization in the Euclidean setting and associates an operator $Op(\sigma)$ to a \textit{symbol} $\sigma$ in the following way: For any $f \in \mathcal{D}(\G)$ and $x \in \G$, the operator

\begin{equation}\label{quant}
Op(\sigma) \phi(x) = \int_{\hat{G}} \text{Tr} \left( \pi(x) \sigma(x,\pi) \hat{\phi}(x) \right)\,d\mu(\pi)\,,
\end{equation} 
where $\mu$ is the Plancherel measure on $\hat{\G}$, is well-defined and continuous. The 
Recall that a \textit{symbol} $\sigma$ is a field of operators $\{\sigma(x,\pi):\mathcal{H}_{\pi}^{\infty} \rightarrow \mathcal{H}_{\pi}, (x,\pi) \in \G \times \hat{\G} \}$, satisfying for each $x \in \G$
\begin{equation}\label{def.symb}
\sigma(x,\cdot):=\{\sigma(x,\pi):\mathcal{H}_{\pi}^{\infty} \rightarrow \mathcal{H}_{\pi},\pi \in \hat{\G}\}\in L^{\infty}_{a,b}(\hat{\G})\,,
\end{equation}
for some $a,b \in \mathbb{R}$, where $L^{\infty}_{a,b}(\hat{\G})$ is as in Definition \ref{Labinf}.\\
Concrete examples of symbols that do not depend on $x \in \G$ are operators of the form $\pi(X)^{\alpha}$, $\alpha \in \mathbb{N}_{0}^{n}$. 

In order to motivate the work presented in this paper, we focus our attention in the symbol classes $S^{m}_{\rho, \delta}= S^{m}_{\rho, \delta}(\G)$ for $(\rho,\delta)$ with $0 \leq \delta \leq \rho \leq 1$ and $m \in \mathbb{R}\cup \{-\infty\}$, introduced in \cite[Section 5.2]{FR16} in the context of a graded Lie group $\G$. An application of the quantization process described in \eqref{quant} then, yields the corresponding operator classes 
\[
\Psi^{m}_{\rho, \delta} = \text{Op} (S^{m}_{\rho,\delta})\,.
\]

More accurately:
 \begin{defn}($\Psi_{\rho,\delta}^{m}(\G)$ symbol classes).\label{Psi.class}
 Let $m, \rho,\delta \in \mathbb{R}$ with $0 \leq \delta \leq \rho \leq 1$, and let $\mathcal{R}$ be a positive Rockland operator of homogeneous degree $\nu$. A symbol in the sense of \eqref{def.symb}, is called a symbol of order $m$ and of type $(\rho,\delta)$ if, for each $\alpha, \beta \in \mathbb{N}_{0}^{n}$ and $\gamma \in \mathbb{R}$, we have
 \[
 \sup_{x \in \G} ||X_{x}^{\beta} \Delta^{\alpha} \sigma(x, \cdot)||_{L^{\infty}_{\gamma, \rho [\alpha]-m-\delta[\beta]+\gamma}(\hat{{\G}})}< \infty \footnote{For non-zero multi-indexes $\alpha=(\alpha_1,\cdots,\alpha_N),\beta=(\beta_1,\cdots,\beta_N)$ and $\nu_1,\cdots,\nu_N$ the dilations' weights of $\G$\,, $[\alpha]:=\nu_1\alpha_1+\cdots+\nu_N \alpha_N$ and $[\beta]:=\nu_1\beta_1+\cdots\nu_N \beta_N$.}\,,
 \]
 where $\Delta^{\alpha}$ are the difference operators in $\G$, discussed extensively later in the Section.
 \end{defn}

Before giving the main properties of the symbol classes $S^{m}_{\rho,\delta}(\G)$, let us state a few remarks:
\begin{rem}\label{rockland}
For the purposes of Definition \ref{Psi.class}, we assume that the Rockland operator $\mathcal{R}$ is fixed. However, the appearing class $S^{m}_{\rho,\delta}(\G)$ exists independently of the choice of $\mathcal{R}$ as Remark \ref{eqiv.norms} suggests.
\end{rem}
\begin{rem}\label{Psi.Euc}
For the case where $\G=(\mathbb{R}^n,+)$, the symbol classes $S^{m}_{\rho, \delta}(\G)$ coincide with the usual H\"{o}rmander classes $S^{m}_{\rho, \delta}(\mathbb{R}^n)$ as in \cite{Hor85}, where the difference operator $\Delta^{\alpha}$ becomes the usual derivatives with respect to the Fourier variable (see also Example \ref{euc.dif}).
\end{rem}
\begin{rem}\label{Psi.inf}
For $m$ as in Definition \ref{Psi.class}, we can allow $m=-\infty$. This can be justified if we set $S^{-\infty} := \bigcap\limits_{m \in \mathbb{R}} S^{m}_{\rho,\delta}$ to denote the class of \textit{smoothing symbols}. In this case the associated operator class is denoted by $\Psi^{-\infty}$.
\end{rem}
The class of symbols $S^{m}_{\rho, \delta}$ fulfills the desired properties of a symbolic calculus. Briefly, the following properties hold:
\begin{itemize}
\item\,$\bigcup\limits_{ m \in \mathbb{R}}\Psi^{m}_{\rho,\delta}$ forms an algebra of operators.
\item If $T$ is in $\Psi^{m}_{\rho, \delta}$, then its formal adjoint $T^{*}$ is also in $\Psi^{m}_{\rho, \delta}$.
\item If $T \in \Psi^{m}_{\rho, \delta}$, then $T$ extends to a continuous operator from $L^{2}_{s}(\G)$ to $L^{2}_{s-m}(\G)$, for any $s \in \mathbb{R}$.
\item If $T \in \Psi^{m}_{\rho, \delta}$ satisfies the ellipticity condition in this setting, then there exists $T^{-1} \in \Psi^{-m}_{\rho, \delta}$ such that 
\[
TT^{-1}-I \in \Psi^{-\infty}\,.
\]
\end{itemize}
One can refer to (\cite[Theorem 5.5.3]{FR16}),(\cite[Theorem 5.5.12]{FR16}), (\cite[Corollary 5.7.2]{FR16}) and (\cite[Theorem 5.8.7]{FR16}), respectively, for proofs of the above properties. 
\begin{rem}
The notion of ellipticity in this setting mentioned in the last property above is not necessary to be given for the purposes of this work, and one can refer to (\cite[Definition 5.8.1]{FR16}) for a precise definition. However, it is worthwhile to mention that the last property also holds for operators that are hypoelliptic within these classes. 
\end{rem}

\subsection{Difference operators}
 
The difference operators in the setting of a graded Lie group are defined as acting on the spaces $\mathcal{K}_{a,b}(\G)$.

Recall that by $\mathcal{K}_{a,b}(\G)$ we denote the subspace of tempered distributions $\kappa \in \mathcal{S}^{'}(\G)$ such that the operator 
\[
\mathcal{S}(\G) \ni \phi \mapsto \phi \ast \kappa\,,
\]
extends  to an operator in $\mathscr{L}(L^{2}_{a}(\G),L^{2}_{b}(\G))$, where the latter stands for the set of linear bounded operators from $L^{2}_{a}(\G)$ to $L^{2}_{b}(\G)$. In view of \cite[Proposition 5.1.24]{FR16}, one can extend the definition of the group Fourier transform to the above space. 

\begin{defn}(The group Fourier transform on $\mathcal{K}_{a,b}(\G)$)
\label{four.transf.}
The group Fourier transform of a tempered distribution $\kappa \in \mathcal{K}_{a,b}(\G)$ is the field of operators $\sigma:=\{\sigma_{\pi}:\mathcal{H}_{\pi}^{\infty} \rightarrow \mathcal{H}_{\pi}^{b}, \pi \in \hat{\G}\} \in L_{a,b}^{\infty}(\hat{\G})$, where for $\pi \in \hat{\G}$ the operator $\sigma_{\pi}:=\pi(\kappa)=\hat{\kappa}(\pi)$ is such that 
\[
\sigma_{\pi}\hat{\phi}(\pi)=\widehat{\phi \ast \kappa}(\pi)\,,\quad \phi \in \mathcal{S}(\G)\,.
\]
\end{defn}
Concrete formulas for the group Fourier transform for the settings we consider in this work, will be given later in Sections \ref{SEC:Groups} and \ref{SEC:Dif}.

Following (\cite[Section 5.2.1]{FR16}), the difference operators in the setting of a graded Lie group are defined as:
\begin{defn}(Difference operators).\label{diff.op1}
For any $q \in C^{\infty}(\G)$, the difference operator associated to $q$, is the operator $\Delta_{q}$ acting by
\[
\Delta_{q}(\hat{f}):=\widehat{qf}(\pi) \equiv \pi(qf)\,,
\]
on any $f \in \mathcal{D}^{'}(\G)$ such that $ f \in \mathcal{K}_{a,b}(\G)$ and $qf \in \mathcal{K}_{a^{'},b^{'}}(\G)$ for some $a,b,a^{'},b^{'} \in \mathbb{R}$.
\end{defn}
Note that the definition of the difference operators in this setting has to be given as each $\Delta_{q}$ is acting on the fields of operators parametrized by $\hat{\G}$. For instance, if $\G=\mathbb{R}^{n}$, then the difference operators considered later in Example \ref{euc.dif} include derivatives in the dual variable.

In addition to the more general definition of a difference operator, the following proposition ensures the existence of a particular case of homogeneous polynomials on $\G$ that give rise to the exact cases of difference operators introduced in (\cite[Section 5.2.1]{FR16}), that appeared earlier in the Definition \ref{Psi.class} of the classes of symbols $S^{m}_{\rho, \delta}$. First let us introduce some necessary notation.\\

Clearly, every polynomial $P$ on  $\G$, or in terms $P \in \mathcal{P}(\G)$, can be written as a finite linear combination
\[
P=\sum_{\alpha \in \mathbb{N}_{0}^{n}} c_{\alpha}x^{\alpha}\,,
\]
with \textit{homogeneous degree} 
\[
D^{\circ}P:=\max\{[\alpha]:\alpha \in \mathbb{N}_{0}^{n}\,,c_{\alpha} \neq 0\}.
\]
If $\G$ is equipped with the dilations' weights $\nu_1,\cdots,\nu_n$, let 
\[
\mathcal{W}:=\{\nu_1 \alpha_1+\cdots+\nu_n \alpha_n : \alpha_1,\cdots,\alpha_n \in \mathbb{N}_{0} \}\,,
\]
be the set of every possible homogeneous degree $[\alpha]$, $ \alpha \in \mathbb{N}_{0}^{n}$. For the case of a stratified Lie group the above set is simply $\mathbb{N}_{0}$.
\begin{prop}\label{hom.pol}
Fix a basis of the Lie algebra of $\G$ and fix $\alpha \in \mathbb{N}_{0}^{n}$. Then, there exists a unique homogeneous polynomial $q_{\alpha}$ of homogeneous degree $[\alpha]$ such that 
\[
X^{\beta} q_{\alpha}(0) = \delta_{\alpha, \beta} = 
\begin{cases}
1, \quad \text{if}\, \beta=\alpha\,,\\
0\,,\quad \text{otherwise}\,,
\end{cases}
\beta \in \mathbb{N}_{0}^{n}\,.
\]
The polynomials $q_{\alpha}, \alpha \in \mathbb{N}_{0}^{n}$, give a basis for $\mathcal{P}(\G)$. Furthermore, for each $M \in \mathcal{W}$ and $\alpha \in \mathbb{N}_{0}^{n}$, the polynomials $q_{\alpha}$, $ [\alpha]=M$ form a basis of $\mathcal{P}_{=M}$, where $\mathcal{P}_{=M}$ denotes the set of polynomials with homogeneous degree $M$.
\end{prop}

\begin{defn}(Difference operators $\Delta^{\alpha}$).\label{diff.op2}
For $\alpha \in \mathbb{N}_{0}^{n}$, the difference operators are 
\[
\Delta^{\alpha} := \Delta_{\tilde{q}_{\alpha}}\,,
\]
where 
\[
\tilde{q}_{\alpha}(x) := q_{\alpha}(x^{-1})\,,
\]
and $q_{\alpha} \in \mathcal{P}_{=[\alpha]}$ is as in Proposition \ref{hom.pol} .
\end{defn}
As the following example shows, the difference operators generalise the notion of the derivative with respect to the Fourier variable in the Euclidean setting.
\begin{exa}[Difference operators on $\mathbb{R}^n$]\label{euc.dif}
Let $\G=(\mathbb{R}^n,+)$. Then $\hat{\mathbb{R}}^{n}$ is isomorphic to $\mathbb{R}^n$, and the polynomials $q_{\alpha}$ are simply the monomials $(\alpha!)^{-1} x^{\alpha}$. In particular, for $\alpha=\alpha_i= \left( (\delta_{i,k})_{1 \leq k \leq n} \right)$, $q_{\alpha_i}$ is the coordinate function $[x]_{i}$ and 
\[
\Delta^{\alpha_{i}} \mathcal{F}_{\mathbb{R}^n} \phi(\xi) = (2 \pi)^{-n/2} \int_{\mathbb{R}^n} e^{-i x \cdot \xi} (-x_1) \phi(x)\,dx= \left( \frac{1}{i} \frac{\partial}{\partial \xi_{i}} \right) \mathcal{F}_{\mathbb{R}^n} \phi(\xi)\,,\ \phi \in \mathcal{S}(\mathbb{R}^n)\,.
\]
More generally, for $\alpha \in \mathbb{N}_{0}^{n}$, the difference operators $\Delta^{\alpha}$ coincide with the operators $D^{\alpha} = \left(\frac{1}{i} \frac{\partial}{\partial \xi} \right)^{\alpha}$.
\end{exa}
\begin{rem}\label{dif.nonlacal}
We notice that in the previous example, if $q$ is not a coordinate function, but any smooth function, then the corresponding difference operator does not have to be local. This is also the case for the difference operators in the setting of the Engel and the Cartan group, see Remark \ref{nonlocal} for a discussion. 
\end{rem}

\section{Preliminaries on the groups}
\label{SEC:Groups}
In this section we expose the main results necessary for our analysis on the groups. As our major source mainly for the description of the Lie algebras and the corresponding Lie groups, and wherever else stated, we have used  \cite[Section 3]{BGR10}.

\subsection{The Engel group}
\label{Engel}
Let $\mathfrak{l}_4=\text{span}\{I_1,I_2,I_3,I_4\}$ be a 3-step nilpotent Lie algebra, whose generators satisfy the non-zero relations:
\[
[I_1,I_2]=I_3\,, \quad [I_1,I_3]=I_4\,.
\]
Thus, $\mathfrak{l}_4$ is graded, as it can be endowed with the  vector space decomposition
\begin{equation}\label{decomp,eng}
\mathfrak{l}_4=V_1\oplus V_2 \oplus V_3,
\end{equation}
where\[
V_1=\text{span}\{I_1,I_2\}\,, V_2=\text{span}\{I_3\}\,, \textrm{and}\,, V_3=\text{span}\{I_4\}\,,
\]
such that $[V_i,V_j] \subset V_{i+j}$.
Observe also that $\mathfrak{l}_4$ is stratified, since $V_1$ generates all of $\mathfrak{l}_4$. Thus, the corresponding Lie group, called the \textit{Engel group} and denoted by $\mathcal{B}_4$ is a homogeneous Lie group, and the natural dilations on its Lie algebra are given by 
\begin{equation}\label{Dil,eng}
D_r(I_1)=rI_1\,, D_r(I_2)=rI_2\,,D_r(I_3)=r^2 I_3\,, \textrm{and}\, D_r(I_4)=r^3 I_4\,, \quad r>0\,.
\end{equation}
In particular, $\mathcal{B}_4$ can be identified with the manifold $\mathbb{R}^4$ endowed with the group law:
\begin{eqnarray*}
\lefteqn{(x_1,x_2,x_3,x_4) \times (y_1,y_2,y_3,y_4)}\\
&:=& (x_1+y_1,x_2+y_2,x_3+y_3-x_1y_2,x_4+y_4+\frac{1}{2}x_1^2y_2-x_1y_3)\,.
\end{eqnarray*}
This identification, in turn, implies that $T_{e}\mathcal{B}_4 \simeq T_{0}\mathbb{R}^4$, where $e$ is the identity element of $\mathcal{B}_4$ and $0$ is the element $(0,0,0,0)$ in $\mathbb{R}^4$. The basis of $\mathfrak{l}^4$, now called the canonical basis, given by \cite[Section 3.2]{BGR10}, or by explicit calculations, consists of the following canonical left invariant vector fields:

\begin{equation}\label{left.inv.eng}
\begin{split}
X_1(x)&= \frac{\partial}{\partial x_1}\,, \quad X_2(x)=\frac{\partial}{\partial x_2}-x_1 \frac{\partial}{\partial x_3}+ \frac{x^{2}_{1}}{2}\frac{\partial}{\partial x_4}\,, \\
            X_3(x)&= \frac{\partial}{\partial x_3}-x_1 \frac{\partial}{\partial x_4}\,,\quad X_4(x)= \frac{\partial}{\partial x_4}\,,
\end{split}
\end{equation}
where $x=(x_1,x_2,x_3,x_4) \in \mathbb{R}^4$, that, as expected, satisfy the above relations. 
 
Additionally, we notice that the map $\text{exp}_{\mathcal{B}_4}$ is the identity map, meaning that, for the elements of $\mathcal{B}_4$, after choosing $\{X_1,X_2,X_3,X_4\}$ as a basis for $\mathfrak{l}_4$, we have the identification
\begin{equation}\label{expB_4}
(x_1,x_2,x_3,x_4)=\text{exp}_{\mathcal{B}_4}(x_1X_1+x_2X_2+x_3X_3+x_4X_4)\,,
\end{equation}
and that we can fix the Lebesgue measure $dx_1dx_2dx_3dx_4$ on $\mathbb{R}^4$, as being the Haar measure on $\mathcal{B}_4$, see \cite[Proposition 1.6.6]{FR16}. Therefore, in what follows we may formulate as
\begin{equation}\label{integ.engel}
\int_{\mathcal{B}_4}\cdots dx_1dx_2dx_3dx_4= \int_{\mathbb{R}^4} \cdots dx_1dx_2dx_3dx_4\,.
\end{equation}

By \eqref{expB_4} we can transport the dilations \eqref{Dil,eng} to the group side, i.e., 
\[
D_r(x_1,x_2,x_3,x_4)=(r x_1,r x_2, r^2 x_3, r^3 x_4)\,, \quad r>0\,,
\]
whereas the homogeneous dimension of the group is $Q_{\mathcal{B}_4}=1+1+2+3=7$.

One could have alternatively chosen the canonical right invariant vector fields, say $\tilde{X}_i\,,i=1,\cdots,4$, for a basis for $\mathfrak{l}_4$. In the following calculations we find the formulas for that basis, that later on will be useful for establishing a framework of difference operators in our setting. For $f \in C^{\infty}(\mathcal{B}_4)$, the identification \eqref{expB_4} allows us to formulate as:
\begin{align}\label{dx1,right,eng}
\tilde{X}_{1}(x)f(x)&=\frac{d}{dt}\Big|_{t=0}f\left( \left(t \cdot \text{exp}_{\mathcal{B}_4}\tilde{X}_1\right) \times x\right)\nonumber\\
&=\frac{d}{dt}\Big|_{t=0}f\left( (t,0,0,0) \times (x_1,x_2,x_3,x_4) \right)\nonumber\\
&=\frac{d}{dt}\Big|_{t=0}f\left(t+x_1,x_2,x_3-tx_2,x_4+\frac{1}{2}t^2x_2-tx_3 \right)\nonumber\\
&=\left(\frac{\partial}{\partial x_1}-x_2 \frac{\partial}{\partial x_3}-x_3 \frac{\partial}{\partial x_4} \right)f(x)\,.
\end{align}
 Analogously one has,
 \begin{align*}
 \tilde{X}_{2}(x)f(x)&=\frac{d}{dt}\Big|_{t=0}f\left( \left(t \cdot \text{exp}_{\mathcal{B}_4}\tilde{X}_2\right) \times x\right)\nonumber\\
&=\frac{d}{dt}\Big|_{t=0}f\left( (0,1,0,0) \times (x_1,x_2,x_3,x_4) \right)\nonumber\\
&=\frac{d}{dt}\Big|_{t=0}f\left(x_1,x_2+t,x_3,x_4\right)\nonumber\\
&=\left(\frac{\partial}{\partial x_2}\right)f(x)\,,
 \end{align*}
 whereas,
  \begin{equation*}\label{dx34.right,engel}
 \tilde{X}_3(x)=\frac{\partial}{\partial x_3}\,,\quad \tilde{X}_4(x)=\frac{\partial}{\partial x_4}\,.
 \end{equation*}
\begin{exa}\label{sub.Engel}
The left-invariant sub-Laplacian on the Engel group $\mathcal{B}_4$ given by 
\begin{align*}
\mathcal{L}_{\mathcal{B}_4}&= X_{1}^{2}+X_{2}^{2}\\
&=\frac{\partial^2}{\partial x_{1}^{2}}+ \left(\frac{\partial}{\partial x_2}-x_1 \frac{\partial}{\partial x_3}+ \frac{x^{2}_{1}}{2}\frac{\partial}{\partial x_4} \right)^2\,,
\end{align*}
is homogeneous of degree $2$ since $D_r(\mathcal{L}_{\mathcal{B}_4})=r^{2} \mathcal{L}_{\mathcal{B}_4}$. The positive Rockland operator in this setting is $\mathcal{R}_{\mathcal{B}_4}:=- \mathcal{L}_{\mathcal{B}_4}$, see \cite[Lemma 4.1.8]{FR16}.
\end{exa}

The next proposition has been proved by Dixmier in \cite[p.333]{Dix57}.
\begin{prop}\label{Dixmier, Engel}
	The dual space of $\mathcal{B}_4$ is $\hat{\mathcal{B}_4}= \{\pi_{\lambda, \mu}\arrowvert \lambda\neq 0, \mu,\lambda \in \mathbb{R}\}$. In particular, for each $(x_1,x_2,x_3,x_4) \in \mathcal{B}_4\,,\pi_{\lambda, \mu}(x_1,x_2,x_3,x_4)$ is acting on $L^2(\mathbb{R},\mathbb{C})$ via
\[
\pi_{\lambda, \mu}(x_1,x_2,x_3,x_4)h(u)\equiv \exp \left(i \left(-\frac{\mu}{2\lambda}x_2+\lambda x_4-\lambda x_3 u + \frac{\lambda}{2}x_2 u^2\right)\right)h(u+x_1)\,.
\]
\end{prop}

	We assume that the Hilbert space $L^2(\mathbb{R},\mathbb{C})$ is endowed with the standard product \[ (h_1,h_2):= \int_{\mathbb{R}} h_1(u)\overline{h_2(u)}\,du\,,\]  where $du$ is the Lebesgue measure on $\mathbb{R}$.

\begin{rem}
\label{smoothvectors}
For the case of the representation $\pi_{\lambda, \mu}$ given above, as well as for any representation acting on some $L^{2}(\mathbb{R}^{m})$ of a connected, simply connected nilpotent Lie group, the space $\mathcal{H}_{\pi}^{\infty}$ is simply the Schwartz space $\mathcal{S}(\mathbb{R}^m)$, see \cite[Corollary 4.1.2]{CG90}.
\end{rem} 

By \cite[Subsection 3.2.3]{BGR10}, or by explicit calculations using \eqref{defn.infin}, the group Fourier transform of the elements of $\mathfrak{l}_4$ are the operators acting on $\mathcal{S}(\mathbb{R})$ given by
\begin{equation}
\label{inf1,2,eng}
\pi_{\lambda, \mu}(X_1)=\frac{d}{du}\,,\quad \pi_{\lambda, \mu}(X_2)=\bigg( -\frac{i\mu}{2\lambda} +\frac{i}{2}\lambda u^2 \bigg )\,.
\end{equation}
Further calculations show that
\begin{equation}
\label{inf3,4,5,eng}
\pi_{\lambda, \mu}(X_3)=-i\lambda u\,,\quad  \pi_{\lambda, \mu}(X_4)=i \lambda\,.
\end{equation}
\begin{exa}
\label{inf,sub,eng}
The group Fourier transform of the canonical sub-Laplacian on the group is, by Remark \ref{smoothvectors}, an operator from $\mathcal{S}(\mathbb{R})$ to $\mathcal{S}(\mathbb{R})$, given by
	\begin{align*}
	\pi_{\lambda, \mu}(\mathcal{L}_{\mathcal{B}_4})=\pi_{\lambda, \mu}(X_1^2+X_2^2)=& \left(\frac{d}{du} \right)^2+ \left( -\frac{i\mu}{2\lambda}+ \frac{i \lambda u^2}{2}\right)^2\\
	&= \frac{d^2}{du^2} -\frac{1}{4} \left(\lambda u^2 -\frac{\mu}{\lambda} \right)^2\,.
	\end{align*}
Observe that $\pi_{\lambda, \mu}(\mathcal{L}_{\mathcal{B}_4})$ is a particular case of the anharmonic oscillator considered in \cite{CDR18}.
\end{exa}
\begin{rem}
\label{globalsymbol}
The group Fourier transform of $\mathcal{L}_{\mathcal{B}_4}$ is also a symbol (independent of $x$) in the sense of \eqref{def.symb}. For instance, the field of operators
\[
\{\pi_{\lambda,\mu}(\mathcal{L}_{\mathcal{B}_4}): \mathcal{S}(\mathbb{R}) \rightarrow \mathcal{S}(\mathbb{R})\,, \pi_{\lambda,\mu} \in \hat{\mathcal{B}}_4\}
\]
is a symbol in $L^{\infty}_{s+2,s}(\hat{\G})$, as a consequence of \cite[Example 5.1.26]{FR16}. Later or, in Section \ref{SEC:multipliers}, we will refer to operators of that form as symbols.
\end{rem}
Due to Proposition \ref{Dixmier, Engel}, the group Fourier transform of a function $f \in L^{1}(\mathcal{B}_4)$ at $\pi_{\lambda, \mu}\in \hat{\mathcal{B}_4}$ is, as follows from \eqref{group,fourier,gen}, given by: 
\begin{equation}
\label{group-fourier-tranform-engel}
\mathcal{F}_{\mathcal{B}_4}(f)(\pi_{\lambda, \mu})\equiv\hat{f}(\pi_{\lambda, \mu})\equiv \pi_{\lambda, \mu}(f):=\int_{\mathcal{B}_4}f(x) \pi_{\lambda, \mu}(x)^{*}dx\,,
\end{equation}
where $x=(x_1,x_2,x_3,x_4) \in \mathbb{R}^4$, is a linear endomorphism of $L^2(\mathbb{R})$, and, after calculations, the adjoint $\pi_{\lambda, \mu}(x)^{*}$ of the unitary operator $\pi_{\lambda,\mu}(x)$ is acting on $h=h(u) \in L^2(\mathbb{R})$ via
\begin{equation}
\label{adjoint_pi,engel}
\pi_{\lambda,\mu}^{*}(x)h(u)=\exp \left( i \left( \frac{\mu}{2 \lambda} x_2 - \lambda x_4+\lambda x_3 (u-x_1)-\frac{\lambda}{2}x_2(u-x_1)^2 \right) \right) h(u-x_1)\,.
\end{equation}
 Rigorous computations show that $\hat{f}(\pi_{\lambda,\mu})h(u)$ can be written as 
\begin{align}
\label{fourier,tansform,engel,pseudo-dif}
\lefteqn{ \int_{\mathbb{R}^4} \bigg[ f(x_1,x_2,x_3,x_4)}\nonumber\\
& \cdot \exp \left( i \left( \frac{\mu}{2 \lambda}x_2 - \lambda x_4 +\lambda x_3 (u-x_1)-\frac{\lambda}{2}x_2 (u-x_1)^{2} \right) \right) h(u-x_1)\bigg] dx_1\,dx_2\,dx_3\,dx_4\tag{1}\\
& = (2 \pi)^{-2}\int_{\mathbb{R}^4} \int_{\mathbb{R}^4} \bigg[ \mathcal{F}_{\mathbb{R}^4}(f) (\xi, \eta, \tau, \omega) \cdot e^{ix_1 \xi}\cdot e^{i x_2 \eta}\cdot e^{i x_3 \tau}\cdot e^{i x_4 \omega} \nonumber\\
& \cdot \exp \left( i \left( \frac{\mu}{2 \lambda}x_2 - \lambda x_4 +\lambda x_3 (u-x_1)-\frac{\lambda}{2}x_2 (u-x_1)^{2} \right) \right) \nonumber\\ & \cdot h(u-x_1)\bigg] dx_1\,dx_2\,dx_3\,dx_4 d\xi\,d\eta\,d\tau\,d\omega \nonumber \\
&= -(2\pi) \int_{\mathbb{R}}\int_{\mathbb{R}} \bigg[ e^{i x_1 \xi} \mathcal{F}(f)_{\mathbb{R}^4} ( \xi, \frac{\lambda}{2}(u-x_1)^2-\frac{\mu}{2 \lambda}, \lambda(x_1-u), \lambda)h(u-x_1) \bigg] dx_1\,d\xi\nonumber\\
&=(2 \pi) \int_{\mathbb{R}} \int_{\mathbb{R}}\bigg[ e^{i(u-v)\xi} \mathcal{F}(f)_{\mathbb{R}^4}(\xi,  \frac{\lambda}{2}v^2-\frac{\mu}{2 \lambda}, -\lambda v, \lambda)h(v)\bigg] dv\,d\xi\nonumber\,,
\end{align}
where for the last inequality we have applied the change of variable  $v=u-x_1$.\\
Here the Fourier transform $\mathcal{F}_{\mathbb{R}^m}(f)(\xi)$ is defined via:
\[
\mathcal{F}_{\mathbb{R}^m}(f)(\xi):= (2 \pi)^{-m/2} \int_{\mathbb{R}^m}f(x)e^{-ix \xi}dx\,,\quad x\,, \xi \in \mathbb{R}^{m}\,,
\]
so that the Fourier inversion theorem becomes
\[
\int_{\mathbb{R}^m}\int_{\mathbb{R}^m}e^{i(u-v)\xi}f(v)\,dv\,d\xi=(2\pi)^m f(u)\,,\quad u,v,\xi \in \mathbb{R}^m\,.
\]
 In particular, the above calculations show that 
 \begin{equation}\label{g.f.t.eng.quant}
 \mathcal{F}_{\mathcal{B}_4}(f)(\pi_{\lambda,\mu})=Op[a_{f,\lambda,\mu}(\cdot,\cdot)]\,,
 \end{equation}
 where 
 \[
 a_{f,\lambda,\mu}(v,\xi)=(2\pi)^2 \mathcal{F}_{\mathbb{R}^4}(f)(\xi,\frac{\lambda}{2}v^2-\frac{\mu}{2\lambda},-\lambda v, \lambda)\,,
\] 
and $Op$ denotes the Kohn-Nirenberg quantization, that is for a smooth symbol $a$ on $\mathbb{R}\times \mathbb{R}$ the operator 
\[
Op(a)f(u)=(2\pi)^{-1}\int_{\mathbb{R}}\int_{\mathbb{R}}e^{i(u-v)\xi}a(v,\xi)f(v)\,dv\,d\xi\,,
\]
 for $f \in \mathcal{S}(\mathbb{R})$ and $u \in \mathbb{R}$.\\
 
Let us finally note that by \eqref{g.f.t.eng.quant}, we see that for a symbol $\sigma$ as in \eqref{def.symb} quantized as:
\[
\sigma(x,\pi_{\lambda,\mu})\equiv \sigma(x,\lambda,\mu)=Op(a_{\kappa_x,\lambda,\mu})\,,
\] 
then its symbol that is given by 
\begin{equation}\label{symbol,lamda,mu}
a_{\kappa_x,\lambda,\mu}(v,\xi)=(2\pi)^2 \mathcal{F}_{\mathbb{R}^4}(\kappa_x)(\xi,\frac{\lambda}{2}v^2-\frac{\mu}{2\lambda},-\lambda v,\lambda)\,,
\end{equation}
where $\{\kappa_x(y)\}$ is the kernel of the symbol $\sigma(x,\lambda,\mu)$, i.e., 
\[
\sigma(x,\lambda,\mu)=\pi_{\lambda,\mu}(\kappa_x)\,.
\]
For our notation the quantization \eqref{quant} becomes
\[
Op(\sigma)\phi(x)=2^{-3}\pi^{-4} \int_{\lambda \neq 0} \int_{\mu \in \mathbb{R}}Tr (\pi_{\lambda,\mu}(x)\sigma(x,\lambda,\mu)\pi_{\lambda,\mu}(\phi))\,d\mu\,d\lambda\,,
\]
where for the Plancherel measure on $\mathcal{B}_4$, see Appendix \ref{sec.foo}, or \cite{Dix57}. Alternatively, by using the property of the Fourier transform 
\[
\hat{\phi}(\pi_{\lambda,\mu})\pi_{\lambda,\mu}(x)=\mathcal{F}_{\mathcal{B}_4}(\phi(x\cdot))(\pi_{\lambda,\mu})\,,
\]
and the properties of the trace we arrive at the equivalent formula
\[
Op(\sigma)\phi(x)=2^{-3}\pi^{-4}\int_{\lambda \neq 0}\int_{\mu \in \mathbb{R}}Tr(Op(a_{\kappa_x,\lambda,\mu})Op(a_{\phi(x\cdot),\lambda,\mu}))\,d\mu\,d\lambda\,,
\]
in terms of composition of quantization of symbols in the Euclidean space. 
\subsection{The Cartan group}
\label{Cartan}
Let $\mathfrak{l}_5=\text{span}\{I_1,I_2,I_3,I_4,I_5\}$ be the 3-step nilpotent Lie algebra, whose generators satisfy the non-zero relations:
\[
[I_1,I_2]=I_3\,, [I_1,I_3]=I_4\,, [I_2,I_3]=I_5\,.
\]
Thus, $\mathfrak{l}_5$ is graded, as it admits a vector space decomposition of the form
 \begin{equation}\label{decomp,cart}
\mathfrak{l}_5=V_1\oplus V_2 \oplus V_3
\end{equation}
where\[
V_1=\text{span}\{I_1,I_2\}\,, V_2=\text{span}\{I_3\}\,, \textrm{and}\,,V_3=\text{span}\{I_4,I_5\}\,,
\]
such that $[V_i,V_j] \subset V_{i+j}$.
Now, observe that $\mathfrak{l}_5$ is stratified, since $V_1$ generates all of $\mathfrak{l}_5$. The corresponding Lie group, called the \textit{Cartan group} and denoted by $\mathcal{B}_5$ is a homogeneous Lie group, and the natural dilations on $\mathfrak{l}_5$ are given by 
\begin{equation*}
D_r(I_1)=rI_1\,, D_r(I_2)=rI_2\,,D_r(I_3)=r^2 I_3\,,D_r(I_4)=r^3I_4\,, \textrm{and}\, D_r(I_4)=r^3 I_5\,, \quad r>0\,.
\end{equation*}
We identify $\mathcal{B}_5$ with the manifold $\mathbb{R}^5$ endowed with the group law 
\begin{equation*}
\begin{split}
 &(x_1,x_2,x_3,x_4,x_5) \times (y_1,y_2,y_3,y_4,y_5):=\\
 &  \left(x_1+y_1,x_2+y_2,x_3+y_3-x_1y_2,x_4+y_4+\frac{x_1^2y_2}{2}-x_1y_3, x_5+y_5 + \frac{x_1y_2^2}{2}-x_2y_3+x_1x_2y_2\right).       
\end{split}
\end{equation*}

Therefore, $T_{e}\mathcal{B}_5 \simeq T_{0}\mathbb{R}^5$, where $e$ is the identity element of $\mathcal{B}_5$ and $0$ is the element $(0,0,0,0,0)$ in $\mathbb{R}^5$. The canonical basis of $\mathfrak{l}^5$, given by \cite[Section 3.3]{BGR10}, or by explicit calculations, consists of the canonical left invariant vector fields

\begin{equation}\label{left.inv.car}
\begin{split}
X_1(x)&=\frac{\partial}{\partial x_1}\,, \quad X_2(x)=\frac{\partial}{\partial x_2}-x_1 \frac{\partial}{\partial x_3} 
+ \frac{x_1^2}{2} \frac{\partial}{\partial x_4}+x_1x_2\frac{\partial}{\partial x_5}\,, \\
            X_3(x)&= \frac{\partial}{\partial x_3}-x_1 \frac{\partial}{\partial x_4}-x_2 \frac{\partial}{\partial x_5}\,,\quad  X_4(x)= \frac{\partial}{\partial x_4}\,,\quad X_5(x)=\frac{\partial}{\partial x_5}\,,
\end{split}
\end{equation}
where $x=(x_1,x_2,x_3,x_4,x_5) \in \mathbb{R}^5$, that satisfy, as one expects, the above relations.

 If we choose the set $\{\tilde{X_1},\tilde{X}_2,\tilde{X}_3, \tilde{X}_4,\tilde{X}_5\}$ of right invariant vector fields as a basis of $\mathfrak{l}_5$, then the later acts on the space $C^{\infty}(\mathcal{B}_5)$ by
\begin{align}\label{dx1,right,cartan}
\tilde{X}_1(x)f(x)&=\frac{d}{dt}\Big|_{t=0}f \left( \left(t \cdot \text{exp}_{\mathcal{B}_5}\tilde{X}_1 \right) \times x \right)\nonumber\\
&=\frac{d}{dt}\Big|_{t=0} f \left( \left(t,0,0,0,0\right) \times \left(x_1,x_2,x_3,x_4,x_5 \right) \right)\nonumber\\
&=\frac{d}{dt}\Big|_{t=0} f(t+x_1,x_2,x_3-tx_2,x_4+\frac{1}{2}t^2x_2-tx_3,x_5+\frac{1}{2}tx_{2}^{2})\nonumber\\
&=\left( \frac{\partial}{\partial{x_1}}-x_{2} \frac{\partial}{\partial{x_3}}-x_3 \frac{\partial}{\partial{x_4}}+\frac{x_{2}^{2}}{2} \frac{\partial}{\partial{x_5}}\right) f(x)\,,
\end{align}
\begin{align*}\label{dx2,right,cartan}
\tilde{X}_2(x)f(x)&=\frac{d}{dt}\Big|_{t=0}f \left( \left(t \cdot \text{exp}_{\mathcal{B}_5}\tilde{X}_2 \right) \times x \right)\nonumber\\
&=\frac{d}{dt}\Big|_{t=0} f \left( \left(0,t,0,0,0\right) \times \left(x_1,x_2,x_3,x_4,x_5 \right) \right)\nonumber\\
&=\frac{d}{dt}\Big|_{t=0} f(x_1,x_2+t,x_3,x_4,x_5-tx_3)\nonumber\\
&=\left(\frac{\partial}{\partial{x_2}}-x_3 \frac{\partial}{\partial{x_5}}\right) f(x)\,,
\end{align*}
whereas similarly one gets,
\begin{equation}\label{dx345.right,cartan}
\tilde{X}_3(x)=\frac{\partial}{\partial x_3}\,,\quad \tilde{X}_4(x)=\frac{\partial}{\partial x_4}\,,\quad \text{and}\,\quad \tilde{X}_5(x)=\frac{\partial}{\partial x_5}\,.
\end{equation}
Note that the map $\text{exp}_{\mathcal{B}_5}$ is the natural diffeomorphism  from $\mathfrak{l}_5$ onto $\mathcal{B}_5$ as in \eqref{expB_4} for the case of $\mathcal{B}_4$, due to \cite[Proposition 1.6.6]{FR16}, yielding analogous results, see\eqref{integ.engel}, concerning the integration on $\mathcal{B}_5$. Finally, transporting the dilations above to the group side, we get
\[
D_r(x_1,x_2,x_3,x_4,x_5)=(rx_1,rx_2,r^2x_3,r^3x_4,r^3x_5)\,.
\] In this case the homogeneous dimension is $Q_{\mathcal{B}_5}=1+1+2+3+3=10$.
\begin{exa}\label{sub.Cartan}
The left-invariant sub-Laplacian on the Cartan group $\mathcal{B}_5$ given by 
\begin{align*}
\mathcal{L}_{\mathcal{B}_5}&= X_{1}^{2}+X_{2}^{2}\\
&=\frac{\partial^2}{\partial x_{1}^{2}}+ \left(\frac{\partial}{\partial x_2}-x_1 \frac{\partial}{\partial x_3} 
+ \frac{x_1^2}{2} \frac{\partial}{\partial x_4}+x_1x_2\frac{\partial}{\partial x_5}\right)^2\,,
\end{align*}
is an operator of homogeneous degree $2$. The positive Rockland operator is, reasoning as before, $\mathcal{R}_{\mathcal{B}_5}:=- \mathcal{L}_{\mathcal{B}_5}$.
\end{exa}

Dixmier in \cite[p.338]{Dix57} showed that,
\begin{prop}\label{Dixmier,cartan}
	The dual space of $\mathcal{B}_5$ is $\hat{\mathcal{B}_5}= \{\pi_{\lambda, \mu, \nu}\arrowvert \lambda^2+\mu^2\neq 0, \nu, \lambda, \mu \in \mathbb{R}\}$. In particular, for each $(x_1,x_2,x_3,x_4,x_5) \in \mathcal{B}_5$ $\,,\pi_{\lambda, \mu, \nu}(x_1,x_2,x_3,x_4,x_5) $ is acting on $L^2(\mathbb{R},\mathbb{C})$ via
\[
\pi_{\lambda, \mu, \nu}(x_1,x_2,x_3,x_4,x_5)h(u)\equiv \exp (i A^{\lambda, \mu, \nu}_{x_1,x_2,x_3,x_4,x_5}(u))h\left( u+\frac{\lambda x_1+\mu x_2}{\lambda^2+\mu^2} \right)\,,
\]
with 
\begin{align*}
A^{\lambda, \mu, \nu}_{x_1,x_2,x_3,x_4,x_5}(u)=& -\frac{1}{2} \frac{\nu}{\lambda^2 +\mu^2}(\mu x_1 -\lambda x_2)+\lambda x_4+ \mu x_5\\
& - \frac{1}{6} \frac{\mu}{\lambda^2 + \mu^2} (\lambda^2 x_1^3+3 \lambda \mu x_1^2 x_2 + 3 \mu^2 x_1 x_2^2 -\lambda \mu x_2^3)\\
& +\mu^2 x_1x_2 u +\lambda \mu (x_1^2-x_2^2) u +\frac{1}{2} (\lambda^2+\mu^2)(\mu x_1-\lambda x_2)u^2\,.
\end{align*}
\end{prop}
As before, the Hilbert space $L^2(\mathbb{R}, \mathbb{C})$ is endowed with the standard product.
Working as we did previously for the case of $\mathcal{B}_4$, the group Fourier transform at $\pi_{\lambda, \mu, \nu}$ acts on the elements $X_1,X_2$ of the canonical basis of $\mathcal{L}_5$ via
\begin{equation}\label{inf2,car}
\begin{split}
\pi_{\lambda, \mu, \nu}(X_1)&=\left(-\frac{i}{2} \frac{\nu \mu}{\lambda^2 + \mu^2} - \frac{i}{2} (\lambda^2 + \mu^2) \mu u^2 + \frac{\lambda}{\lambda^2 + \mu^2} \frac{d}{du} \right)\,,\\
\pi_{\lambda, \mu, \nu}(X_2)&=\left( \frac{i}{2} \frac{ \lambda \nu}{\lambda^2 + \mu^2}+\frac{i}{2} (\lambda^2 + \mu^2) \lambda u^2+\frac{\mu}{\lambda^2 + \mu^2} \frac{d}{du} \right)\,,
\end{split}
\end{equation}
see \cite[Subsection 3.3.2]{BGR10}, or after explicit calculations using \eqref{defn.infin}.

Further calculations show that
\begin{equation}
\label{inf3,4,5,car}
\pi_{\lambda, \mu, \nu}(X_3)=i (\lambda^2 + \mu^2)u\,,\quad \pi_{\lambda, \mu, \nu}(X_4)=i \lambda \\,\quad  \pi_{\lambda, \mu, \nu}(X_5)=i \mu\,.
\end{equation}
\begin{exa}
\label{inf,sub,eng}
The group Fourier transform at $\pi_{\lambda, \mu, \nu}$ of the canonical sub-Laplacian $\mathcal{L}_{\mathcal{B}_5}$ is the operator given by
	\begin{align*}
	\pi_{\lambda, \mu}(\mathcal{L}_{\mathcal{B}_5})=\pi_{\lambda, \mu}(X_1^2+X_2^2)&=\left(-\frac{i}{2} \frac{\nu \mu}{\lambda^2 + \mu^2} - \frac{i}{2} (\lambda^2 + \mu^2) \mu u^2 + \frac{\lambda}{\lambda^2 + \mu^2} \frac{d}{du} \right)^2\\&+ \left( \frac{i}{2} \frac{\lambda \nu}{\lambda^2 + \mu^2}+\frac{i}{2} (\lambda^2 + \mu^2) \lambda u^2+\frac{\mu}{\lambda^2 + \mu^2} \frac{d}{du} \right)^2\\
	&=\frac{1}{\lambda^2+ \mu^2} \frac{d^2}{du^2}-\frac{1}{4(\lambda^2 + \mu^2)}((\lambda^2 + \mu^2)^2 u^2 + \nu)^2\,,
	\end{align*}
and in particular, after suitable rescaling, the induced operator on $\mathcal{S}(\mathbb{R})$, takes the form of an anharmonic oscillator with terms $u^2$ and $u^4$. 
\end{exa}
Due to Proposition \ref{Dixmier,cartan} and \eqref{group,fourier,gen}, the group Fourier transform of a function $f \in L^{1}(\mathcal{B}_5)$ at $\pi_{\lambda, \mu, \nu} \in \hat{\mathcal{B}_5}$, is given by
\begin{equation}
\label{group-fourier-transform-cartan}
\mathcal{F}_{\mathcal{B}_5}(f)(\pi_{\lambda, \mu, \nu}) \equiv \hat{f}(\pi_{\lambda, \mu, \nu})\equiv \pi_{\lambda, \mu, \nu}(f):= \int_{\mathcal{B}_5} f(x) \pi_{\lambda, \mu, \nu}(x)^{*}dx\,,
\end{equation}
where $x=(x_1,x_2,x_3,x_4,x_5)\in \mathbb{R}^5$, is a linear endomorphism on $L^2(\mathbb{R})$, where, after calculations, the adjoint $\pi_{\lambda, \mu,\nu}(x)^{*}$ of the unitary operator $\pi_{\lambda,\mu,\nu}(x)$ is acting on $h \in L^2(\mathbb{R})$ via 
\begin{equation*}
\label{adjoin_pi,cartan}
\pi_{\lambda, \mu, \nu}(x)^{*}h(u)=\exp\left\lbrace i \left( A_{x_1,x_2,x_3,x_4,x_5}^{\lambda, \mu, \nu}\right)(u)^{*}\right\rbrace h\left(u-\frac{\lambda x_1 +\mu x_2}{\lambda^2+\mu^2} \right)\,,
\end{equation*}
with 
\begin{align*}
\left( A_{x_1,x_2,x_3,x_4,x_5}^{\lambda, \mu, \nu}\right)(u)^{*}&=\frac{1}{2}\frac{\nu}{\lambda^2+\mu^2}(\mu x_1 - \lambda x_2 )-\lambda x_4 -\mu x_5\nonumber \\
 &+ \frac{\mu}{6(\lambda^2+\mu^2)}\left(\lambda^2 x_1^3 + 3 \lambda \mu x_1^2 x_2+ 3 \mu^2 x_1 x_2^2 - \lambda \mu x_2^3\right)\nonumber\\
 &-\mu^{2}x_1 x_2 \left(u - \frac{\lambda x_1 + \mu x_2}{\lambda^2 + \mu^2}  \right)- \lambda \mu (x_{1}^{2}-x_{2}^{2}) \left( u - \frac{\lambda x_1 + \mu x_2}{\lambda^2+\mu^2} \right)\nonumber\\
 & - \frac{1}{2} (\lambda^2 +\mu^2) (\mu x_1-\lambda x_2) \left(u - \frac{\lambda x_1+ \mu x_2}{\lambda^2 + \mu^2} \right)^{2}\,.
\end{align*}
We note that, unlike the case of the Engel group $\mathcal{B}_4$, the group Fourier transform in this case of the Cartan group $\mathcal{B}_5$, cannot be entirely expressed via a pseudo-differential form. 
\section{On the Construction of $\Psi^{m}_{\rho,\delta}$}
\label{SEC:Dif}
For the construction of the classes $S^{m}_{\rho, \delta}(\G)$ of symbols for any graded Lie group $\G$, the main obstacle one has to overcome is to find explicit expression for the difference operators in each particular setting; the Engel and the Cartan group. \\
In particular, we find the expressions for the difference operators of the form $\Delta_{x_i}$ in our settings. Difference operators of this form have already been given in the graded case for the Heisenberg group $\mathbb{H}_n$ in \cite[Sections 6.3.1, 6.3.2]{FR16}. To justify our choice of difference operators, notice that in Definition \ref{Psi.class} we could have used instead the difference operators $\Delta_{x^{\alpha}}$, where $x^{\alpha}$ is defined by 
\[
x^{\alpha}:=x_{1}^{\alpha_1}\cdots x_{n}^{\alpha_n}\,,
\]
for $\alpha=(\alpha_1, \cdots, \alpha_n) \in \mathbb{N}_{0}^{n}$, see the discussion that follows after \cite[Remark 5.2.13]{FR16}. Additionally, \cite[Proposition 5.2.10]{FR16} yields
\[
\Delta_{x^{\alpha}}= \left(\Delta_{x_1} \right)^{\alpha_1} \cdots \left( \Delta_{x_n} \right)^{\alpha_n}\,.
\]
Thus, knowing explicitly the difference operators of the form $\Delta_{x_i}$, we get sufficient conditions for an operator on $\mathcal{S}(\G)$ to belong 
to the class of operators $\cup_{ m \in \mathbb{R}}\Psi^{m}_{\rho,\delta}(\G)$, when $\G$ is either the Engel group $\mathcal{B}_4$ or the Cartan group $\mathcal{B}_5$. For the moment, let us give some examples of operators that belong to the classes of operators $\cup_{ m \in \mathbb{R}}\Psi^{m}_{\rho,\delta}(\G)$ as well as some counter examples.

Note that \cite[Remark 5.2.18]{FR16} implies that any left-invariant differential operator belongs to $\cup_{m \in \mathbb{R}}\Psi_{\rho, \delta}^{m}(\G)$, for any graded Lie group $\G$. However, as the next examples verifies, this in not the case for every right-invariant operators on the Engel group $\mathcal{B}_4$. Similarly, one can construct right invariant vector fields that do not belong to any $\Psi^{m}$ for the case of the Cartan group $\mathcal{B}_5$.
\begin{exa}[The right invariant vector fields are not necessarily in some $\Psi^{m}$]
By \cite[Corollary 5.7.2]{FR16}, if $A \in \Psi^{m}$ is an operator on $\G$ of homogeneous degree $\nu_{A}$, then 
\begin{equation}
\label{right,inv,counterexample}
\forall s \in \mathbb{R} \quad \exists C>0 \quad \forall f \in \mathcal{S}(\G) \quad ||Af||_{L^{2}_{s-m}} \leq C ||f||_{L^{2}_{s}}\,,
\end{equation}
for $m \geq \nu_{A}$. Consider now the case of the Engel group $\mathcal{B}_4$, and fix the canonical base $\{X_1,X_2,X_3,X_4\}$, with $X_i$'s as in \eqref{left.inv.eng}, of its Lie algebra. Recall that $X_1=\partial_{x_1}$, $\tilde{X}_1=\partial_{x_1}-x_2\partial_{x_3}-x_3 \partial_{x_4}
$, and thus $X_1-\tilde{X}_1=x_2\partial_{x_3}+x_3 \partial_{x_4}$. Using \cite[Corollary 4.13]{Fol75} for the particular case of $\mathcal{B}_4$ we have,
\begin{equation}\label{Folland, right inv}
L^{2}_{1}(\mathcal{B}_4):= \{ f \in L^2(\mathcal{B}_4) : X_1f,X_2f \in L^{2}(\mathcal{B}_4)\}\,.
\end{equation}
Let $\phi,\chi \in C^{\infty}(\mathbb{R})$, such that $\phi, \phi^{'} \in L^2(\mathbb{R})$, and $\chi: \mathbb{R} \rightarrow [0,1]$ be compactly supported. Then, define $f \in C^{\infty}(\mathcal{B}_{4})$ via
\[
f(x_1,x_2,x_3,x_4):=\chi(x_1) \chi(x_3) \chi(x_4) \phi(x_2)\,.
\]
Then, it can be easily checked that $f,X_1 f$ and $X_2 f$ are square integrable, and so, using \eqref{Folland, right inv}, $f \in L_{1}^{2}(\mathcal{B}_4)$. However, since $x_2 \partial_{x_3}f \notin L^2(\mathcal{B}_4) $, then also $(X_1-\tilde{X}_1)f \notin L^2(\mathcal{B}_4) $, and in particular $\tilde{X}_1 f \notin L^2(\mathcal{B}_4)$. 
Assume now that $\tilde{X}_1$ is in some $\Psi^{m}\,, m \geq 1$. Then, an application of \eqref{right,inv,counterexample} yields that $\tilde{X}_1$ maps $L^{2}_{1}(\mathcal{B}_4)$ to some $L^{2}_{1-m}(\mathcal{B}_4)$, $m \geq 1$, and hence to $L^2(\mathcal{B}_4)$, continuously. Thus, the chosen $f$ shows that the right invariant operator $\tilde{X}_{1}$ is not in $\cup_{ m \in \mathbb{R}}\Psi^{m}_{\rho,\delta}(\mathcal{B}_4)$.
\end{exa}

Finally, we note that in the process of finding the formulas for the difference operators in the subsections that follow instead of making direct calculations, we adopt other, more efficient, techniques. In some of those, we make use of the subsequent properties without mentioning:
\begin{itemize}
\item The linearity of the group Fourier transform.
\item For $X$ and $\tilde{X}$ being a left and a right invariant vector field, respectively, in the Lie algebra of each group, and for a distribution $\kappa$ we have
\[
\pi(X\kappa)=\pi(X)\pi(\kappa)\,, \quad \textrm{and}\quad \pi(\tilde{X}\kappa)=\pi(\kappa) \pi(X)\,,
\]
see \cite[Proposition 1.7.6(iv)]{FR16}.
\end{itemize}
\subsection{The case of the Engel group}
The difference operators on a graded Lie group, have been defined as acting on the space of tempered distributions $\mathcal{K}_{a,b}$, see Definition \ref{diff.op1}. Thus, it makes sense to clarify this notion in the particular setting of $\mathcal{B}_4$:

Let $\kappa$ be in some $\mathcal{K}_{a,b}(\mathcal{B}_4)$, and let $\pi_{\lambda, \mu} \in \hat{\mathcal{B}_4}$, where $\hat{\mathcal{B}_4}$ is as in Proposition \ref{Dixmier, Engel}. Then, the group Fourier transform of $\kappa$ is by Definition \ref{group,fourier,gen}, and by Remark \ref{smoothvectors}, the field of operators 
\[
\{ \pi_{\lambda, \mu}(\kappa) : \mathcal{S}(\mathbb{R}) \rightarrow \mathcal{H}_{\pi_{\lambda, \mu}}^{b}\,,\pi_{\lambda,\mu} \in \hat{\mathcal{B}_4} \}\,,
\]
in $L^{\infty}_{a,b}(\hat{\mathcal{B}_4})$, where $\pi_{\lambda,\mu}(\kappa)$ acts on $\mathcal{S}(\mathbb{R})$ as in \eqref{group-fourier-tranform-engel}.

\begin{prop}\label{dx1,dx2,eng}
	 Let $\kappa \in \mathcal{K}_{a,b}(\mathcal{B}_4)$ be a tempered distribution, such that $x_1 \kappa$ is in some $\mathcal{K}_{a^{'},b^{'}}(\mathcal{B}_4)$, or $x_2 \kappa$ is in some $\mathcal{K}_{a^{'},b^{'}}(\mathcal{B}_4)$, so that the difference operator $\Delta_{x_1}$ or $\Delta_{x_2}$, respectively, make sense. Then the formulas for the above difference operators, acting on the space $\mathcal{S}(\mathbb{R})$, are given by: 
	 \[
	 \Delta_{x_1} \hat{\kappa}(\pi_{\lambda, \mu})= \frac{i}{ \lambda} \left(\pi_{\lambda,\mu}(X_3)\pi_{\lambda,\mu}(\kappa)-\pi_{\lambda, \mu}(\kappa)\pi_{\lambda, \mu}(X_3)\right)\,,
	 \]
	 where $\pi_{\lambda, \mu}(X_3)=-i \lambda u$, and
	 \[
	 \Delta_{x_2} \hat{\kappa}(\pi_{\lambda, \mu})= \frac{2 \lambda}{i} \partial_{\mu}\pi_{\lambda, \mu}(\kappa)\,.
	 \]
\end{prop}
\begin{proof}
	Observe that by \eqref{left.inv.eng} and \eqref{dx34.right,engel} we get 
	\begin{equation}\label{observe1}
	\tilde{X}_3-X_3= x_1 \frac{\partial}{\partial x_4}=\frac{\partial}{\partial{x_4}}x_1\,.
	\end{equation}
	Observe also that for any suitable $\kappa$, using \eqref{inf3,4,5,eng} we have
	\begin{equation}\label{forX4}
	\pi_{\lambda, \mu}\left(\frac{\partial}{\partial x_4} \kappa\right)=\pi_{\lambda, \mu}( X_4 \kappa)=\pi_{\lambda, \mu}(X_4)\pi_{\lambda, \mu}(\kappa)=i \lambda \pi_{\lambda, \mu}(\kappa)\,,
	\end{equation}
	since $X_4=\frac{\partial}{\partial_{x_4}}$. Then, by the above,
	\begin{align*}
	\pi_{\lambda, \mu}(x_1 \kappa) &= \frac{1}{i \lambda} \pi_{\lambda, \mu} ( X_4 x_1 \kappa)=\frac{1}{i \lambda} \pi_{\lambda, \mu}((\tilde{X}_3 -X_3)\kappa)\\
	&= \frac{1}{i \lambda} (\pi_{\lambda, \mu}(\tilde{X}_3 \kappa)-\pi_{\lambda, \mu}(X_3 \kappa))\\
	&= \frac{1}{i \lambda} (\pi_{\lambda, \mu}(\kappa)\pi_{\lambda, \mu}(X_3)-\pi_{\lambda, \mu}(X_3)\pi_{\lambda, \mu}(\kappa))\\
	&= \frac{i}{ \lambda} (\pi_{\lambda,\mu}(X_3)\pi_{\lambda,\mu}(\kappa)-\pi_{\lambda, \mu}(\kappa)\pi_{\lambda, \mu}(X_3))\,.	
	\end{align*}
	Now, for the difference operator corresponding to $x_2$, we differentiate the group Fourier transform of $\kappa$ given in \eqref{group-fourier-tranform-engel}, at $h \in \mathcal{S}(\mathbb{R})$ with respect to $\mu$ and get
	
	\begin{align*}
	\partial_{\mu} \{ \pi_{\lambda, \mu}(\kappa)h(u)\}&=\partial_{\mu} \Big\{\int_{\mathbb{R}^4}\kappa(x)\exp\left(i\left(\frac{\mu}{2\lambda}x_2-\lambda x_4 \right)\right)\\
	& \cdot \exp\left(i\left(\lambda x_3 (u-x_1)-\frac{\lambda}{2}x_2 (u-x_1)^2\right)\right)h(u-x_1)dx \Big\}\\
	&=\int_{\mathbb{R}^4}\kappa(x)\exp\left(i\left(\frac{\mu}{2\lambda}x_2-\lambda x_4 \right)\right)\\
	& \cdot \exp\left(i\left(\lambda x_3 (u-x_1)-\frac{\lambda}{2}x_2 (u-x_1)^2\right)\right)h(u-x_1)\left( \frac{i}{2\lambda} x_2 \right)dx\,,
		\end{align*} 
		or in terms of difference operators,
		\[
			\partial_{\mu} \pi_{\lambda, \mu}(\kappa)=\pi_{\lambda, \mu} \left(\frac{i}{2\lambda}x_2 \kappa \right)=\frac{i}{2 \lambda} \Delta_{x_2}\pi_{\lambda, \mu}(\kappa)\,.
		\]
		The proof is complete.
	\end{proof}
	\begin{rem}
\label{nonlocal}
As we mentioned earlier, see Remark \ref{dif.nonlacal}, the difference operators are not necessarily local, in the sense that, as the last Proposition shows, they might for example be expressed by using derivatives in $\mu \in \mathbb{R}$, and therefore they should act on the field of operators $\{ \pi_{\lambda, \mu}(\kappa), \lambda \neq 0, \mu \in \mathbb{R}\}$. This is also the case for the difference operators $\Delta_{x_3|\pi_{\lambda, \mu}}$ and  $\Delta_{x_4|\pi_{\lambda, \mu}}$ in the setting of $\mathcal{B}_4$, see Propositions \ref{dx3,engel} and \ref{dx4,eng}, whereas in the setting of $\mathcal{B}_5$ any difference operator of the form $\Delta_{x_i|\pi_{\lambda, \mu, \nu}}$,$i=1,\cdots,4$ is not local, see Propositions \ref{dx1dx2,cartan}, \ref{dx3,car}, \ref{dx4,cartan} and \ref{dx5,cartan}. 
\end{rem}
	\begin{prop}
	\label{dx3,engel}
	Let $\kappa \in \mathcal{K}_{a,b}(\mathcal{B}_4)$ be such that $x_3 \kappa$ is also in some $\mathcal{K}_{a^{'},b^{'}}(\mathcal{B}_4)$. Then the difference operator that corresponds to $x_3$ is given by
	\[
	\Delta_{x_3} \hat{\kappa}(\pi_{\lambda, \mu})= \frac{i}{ \lambda} \left(\Delta_{x_2}\pi_{\lambda, \mu}(\kappa)\pi_{\lambda, \mu}(X_3)+\pi_{\lambda, \mu}(\kappa) \pi_{\lambda, \mu}(X_1)-\pi_{\lambda, \mu}(X_1)\pi_{\lambda,\mu}(\kappa) \right)\,,
	\]
	where $\Delta_{x_2|\pi_{\lambda, \mu}}$ is given in Proposition \ref{dx1,dx2,eng}, and $\pi_{\lambda, \mu}(X_1), \pi_{\lambda, \mu}(X_3)$, are as in \eqref{inf1,2,eng} and \eqref{inf3,4,5,eng}.
	\end{prop}
\begin{proof}
Formulas \eqref{left.inv.eng}, \eqref{dx1,right,eng}, and \eqref{dx34.right,engel} yield
\[
X_1-\tilde{X}_1-x_2 \tilde{X}_3= x_3\frac{\partial}{\partial{x_4}}=\frac{\partial}{\partial{x_4}} x_3\,.
\]
Let $\kappa$ be a suitable distribution. Then, using \eqref{forX4}, and as $\frac{\partial}{\partial{x_4}}=X_4$, we have
\begin{align*}
\pi_{\lambda, \mu}(x_3 \kappa)&= \frac{1}{i \lambda}\pi_{\lambda, \mu}(X_4 x_3 \kappa)= \frac{1}{i \lambda}\pi_{\lambda, \mu}\left((X_1-\tilde{X}_1-x_2 \tilde{X}_3)\kappa \right)\\
&=\frac{1}{i \lambda}\left(\pi_{\lambda, \mu}(X_1 \kappa)-\pi_{\lambda, \mu}(\tilde{X}_1 \kappa)-\pi_{\lambda, \mu}(x_2 \tilde{X}_3 \kappa)\right)\\
&=\frac{1}{i \lambda}\left(\pi_{\lambda, \mu}(X_1)\pi_{\lambda, \mu}(\kappa)-\pi_{\lambda, \mu}(\kappa)\pi_{\lambda, \mu}(X_1)-\pi_{\lambda, \mu}(x_2\kappa)\pi_{\lambda, \mu}(X_3)\right)\\
&= \frac{1}{i \lambda} \left(\pi_{\lambda, \mu}(X_1) \pi_{\lambda, \mu}(\kappa)- \pi_{\lambda, \mu}(\kappa) \pi_{\lambda, \mu}(X_1)-\Delta_{x_2}\pi_{\lambda, \mu}(\kappa)\pi_{\lambda, \mu}(X_3) \right)\\
&= \frac{i}{ \lambda} \left(\Delta_{x_2}\pi_{\lambda, \mu}(\kappa)\pi_{\lambda, \mu}(X_3)+\pi_{\lambda, \mu}(\kappa) \pi_{\lambda, \mu}(X_1)-\pi_{\lambda, \mu}(X_1)\pi_{\lambda,\mu}(\kappa) \right)
\,,
\end{align*}
completing the proof.
\end{proof}\begin{prop}\label{dx4,eng}
	Let $\kappa \in \mathcal{K}_{a,b}(\mathcal{B}_4)$ be such that $x_4 \kappa$ is also in some $\mathcal{K}_{a^{'},b^{'}}(\mathcal{B}_4)$. Then, for all $h \in \mathcal{S}(\mathbb{R})$ we have
	\begin{align*}
	(\Delta_{x_4} \pi_{\lambda, \mu}(\kappa))h(u)&= i\partial_{\lambda}\{\pi_{\lambda, \mu}(\kappa)h(u)\} -\left(\frac{\mu}{2 \lambda^2}+\frac{u^2}{2} \right)\Big \{\Delta_{x_2}\pi_{\lambda, \mu}(\kappa)h(u) \Big\}\\
	&+u\Big \{ \Delta_{x_3}\pi_{\lambda, \mu}(\kappa)h(u)\Big \}-\Big \{ \Delta_{x_3}\Delta_{x_1}\pi_{\lambda, \mu}(\kappa)h(u)\Big\}\\
	&+u\Big \{ \Delta_{x_2}\Delta_{x_1}\pi_{\lambda, \mu}(\kappa)h(u) \Big \}-\frac{1}{2}\Big\{\Delta_{x_2}\Delta_{x_1}^{2}\pi_{\lambda, \mu}(\kappa)h(u)\Big\}\,,
	\end{align*}
	where the difference operators $\Delta_{x_i|\pi_{\lambda, \mu}}$, $i=1,2,3$, are given in Propositions \ref{dx1,dx2,eng} and \ref{dx3,engel}.
\end{prop}
\begin{proof}
	For $\kappa$ as in the statement and for $h \in \mathcal{S}(\mathbb{R})$, we differentiate with respect to $\lambda$ the group Fourier transform of $\kappa$, given in \eqref{group-fourier-tranform-engel}, at $h$. This yields,
	\begin{align*}
	\partial_{\lambda}\{\pi_{\lambda, \mu}(\kappa)h(u)\}&= \partial_{\lambda} \Big \{ \int_{\mathbb{R}^4} \kappa(x) \exp \left( i \left( \frac{\mu}{2 \lambda}x_2 -\lambda x_4\right) \right) \\
	&\cdot \exp \left(i \left(\lambda x_3(u-x_1)-\frac{\lambda}{2}x_2(u-x_1)^2 \right) \right)h(u-x_1)\,dx \Big\}\\
	&=\int_{\mathbb{R}^4}\kappa(x) \exp \left( i \left( \frac{\mu}{2 \lambda}x_2 -\lambda x_4+\lambda x_3(u-x_1)-\frac{\lambda}{2}x_2(u-x_1)^2 \right) \right)\\
	& h(u-x_1) \Big\{i\left(-\frac{\mu}{2 \lambda^2}x_2-x_4+x_3(u-x_1)-\frac{x_2}{2}(u-x_1)^2 \right) \Big\} dx\,.
	 	\end{align*}
	Rewriting the above formula in terms of difference operators we obtain
	\begin{align*}
	\partial_{\lambda}\{\pi_{\lambda, \mu}(\kappa)h(u)\}&= i \Big[ -\left(\frac{\mu}{2 \lambda^2}+\frac{u^2}{2} \right)\Big \{\Delta_{x_2}\pi_{\lambda, \mu}(\kappa)h(u) \Big\}-\Big \{\Delta_{x_4}\pi_{\lambda, \mu}(\kappa)h(u)\Big\}\\
	&+u\Big \{ \Delta_{x_3}\pi_{\lambda, \mu}(\kappa)h\Big \}(u)-\Big \{ \Delta_{x_3}\Delta_{x_1}\pi_{\lambda, \mu}(\kappa)h(u)\Big\}\\
	&+u\Big \{ \Delta_{x_2}\Delta_{x_1}\pi_{\lambda, \mu}(\kappa)h(u) \Big \}-\frac{1}{2}\Big\{\Delta_{x_2}\Delta_{x_1}^{2}\pi_{\lambda, \mu}(\kappa)h(u)\Big\}\Big]\,,
	\end{align*}
completing the proof.	
	\end{proof}
\subsection{The case of the Cartan group}
\label{Cartan}
Let $\kappa$ be in some $\mathcal{K}_{a,b}(\mathcal{B}_5)$, and let $\pi_{\lambda, \mu, \nu} \in \hat{\mathcal{B}_5}$, where $\hat{\mathcal{B}_5}$ is as in Proposition \ref{Dixmier,cartan}. Then, the group Fourier transform of $\kappa$ is by Definition \ref{group,fourier,gen}, and by Remark \ref{smoothvectors}, the field of operators 
\[
\{ \pi_{\lambda, \mu,\nu}(\kappa) : \mathcal{S}(\mathbb{R}) \rightarrow \mathcal{H}_{\pi_{\lambda, \mu, \nu}}^{b}\,, \pi_{\lambda, \mu, \nu} \in \hat{\mathcal{B}_5} \}\,,
\]
in $L_{a,b}^{\infty}(\hat{\mathcal{B}_5})$, where $\pi_{\lambda, \mu, \nu}(\kappa)$ acts on $\mathcal{S}(\mathbb{R})$ as in \eqref{group-fourier-transform-cartan}.

\begin{prop}\label{dx1dx2,cartan} Let $\kappa \in \mathcal{K}_{a,b}(\mathcal{B}_5)$ be a tempered distribution, such that $x_1 \kappa$ is in some $\mathcal{K}_{a^{'},b^{'}}(\mathcal{B}_5)$, or $x_2\kappa$ is in some $\mathcal{K}_{a^{'},b^{'}}(\mathcal{B}_5)$, so that the difference operator $\Delta_{x_1}$ or $\Delta_{x_2}$, respectively, make sense. Then, the above difference operators act on $h \in \mathcal{S}(\mathbb{R})$ via
\begin{align*}
(\Delta_{x_1} \pi_{\lambda, \mu, \nu} (\kappa))h(u)&=- 2 i\mu \frac{\partial}{\partial \nu}\{\pi_{\lambda, \mu, \nu}(\kappa)h(u)\}\nonumber\\
&+\frac{i \lambda}{(\lambda^2+\mu^2)}\left( \pi_{\lambda, \mu, \nu}(X_3)\pi_{\lambda, \mu, \nu}(\kappa)-\pi_{\lambda, \mu, \nu}(\kappa)\pi_{\lambda, \mu, \nu}(X_3) \right)h(u)\,,
\end{align*}
and 
\begin{align*}
(\Delta_{x_2} \pi_{\lambda, \mu, \nu} (\kappa))h(u)&=2 i\lambda  \frac{\partial}{\partial \nu}\{\pi_{\lambda, \mu, \nu}(\kappa)h(u)\}\nonumber\\
&+i\frac{\mu}{\lambda^2+\mu^2} \left(\pi_{\lambda,\mu,\nu}(X_3) \pi_{\lambda, \mu, \nu}(\kappa)-\pi_{\lambda, \mu, \nu}(\kappa)\pi_{\lambda, \mu, \nu}(X_3) \right)h(u)\,,
\end{align*}
where $\pi_{\lambda,\mu,\nu}(X_3)$ is given by \eqref{inf3,4,5,car}.
\end{prop}
\begin{proof}
We proceed by finding the explicit formulas for the above operators simultaneously, by solving a system of equations. In particular, for $\kappa$ as in the statement, and $h \in \mathcal{S}(\mathbb{R})$, we start by differentiating with respect to $\nu$ the group Fourier transform of $\kappa$ given in \eqref{group-fourier-transform-cartan}. Then,
\begin{align}\label{dnu}
\partial_{\nu}\{\pi_{\lambda,\mu,\nu}(\kappa)h(u)\}&=\int_{\mathbb{R}^5} \kappa(x) \frac{\partial}{\partial \nu} \left[e^{iA^{*}} \cdot h \left(u-\frac{\lambda x_1+\mu x_2}{\lambda^2+\mu^2}\right) \right]\,dx\nonumber\\
&=\int_{\mathbb{R}^5} \kappa(x)\cdot e^{iA^{*}}\cdot \frac{\partial(i A^{*})}{\partial \nu}\cdot h\left(u-\frac{\lambda x_1+\mu x_2}{\lambda^2+\mu^2}\right)\,dx\nonumber\\
&=\frac{i \mu}{2(\lambda^2+\mu^2)}\int_{\mathbb{R}^5} x_1 \kappa(x) \cdot h\left(u-\frac{\lambda x_1+\mu x_2}{\lambda^2+\mu^2}\right)\,dx\nonumber\\
&-\frac{i \lambda}{2(\lambda^2+\mu^2)} \int_{\mathbb{R}^5} x_2 \kappa(x) \cdot h\left(u-\frac{\lambda x_1+\mu x_2}{\lambda^2+\mu^2}\right)\,dx\nonumber\\
&=\frac{i \mu}{2(\lambda^2+\mu^2)}\pi_{\lambda,\mu,\nu}(x_1 \kappa)h(u)-\frac{i\lambda}{2(\lambda^2+\mu^2)}\pi_{\lambda,\mu,\nu}(x_2 \kappa)h(u)\,.
\end{align}
Now, observe that by \eqref{left.inv.car} and \eqref{dx345.right,cartan}, we have that 
\[
\tilde{X}_{3}-X_{3}=x_1X_4+x_2X_5=X_4x_1+X_5x_2\,.
\]
Then, for any suitable $\kappa$ as in the statement, as $X_4=\frac{\partial}{\partial x_4}$ and $X_5=\frac{\partial}{\partial{x_5}}$, since $\pi_{\lambda,\mu,\nu}(X_4)=i \lambda$ and $\pi_{\lambda,\mu,\nu}(X_5)=i \mu$, one gets
\[
\pi_{\lambda, \mu, \nu} \left(\frac{\partial}{\partial{x_4}} \kappa \right)=\pi_{\lambda, \mu, \nu}(X_4 \kappa)=\pi_{\lambda, \mu, \nu}(X_4) \pi_{\lambda, \mu, \nu} (\kappa)= i \lambda \pi_{\lambda, \mu, \nu}(\kappa)\,,
\]
and
\[
\pi_{\lambda, \mu, \nu} \left(\frac{\partial}{\partial{x_5}} \kappa \right)=\pi_{\lambda, \mu, \nu}(X_5 \kappa)=\pi_{\lambda, \mu, \nu}(X_5) \pi_{\lambda, \mu, \nu} (\kappa)= i \mu \pi_{\lambda, \mu, \nu}(\kappa)\,.
\]
Assume first that $\mu \neq 0$. The above observations yield
\begin{align*}
\pi_{\lambda, \mu, \nu}(x_2 \kappa)h(u)&=\frac{1}{i \mu}\pi_{\lambda, \mu, \nu}\left(X_5x_2\kappa\right)h(u)\\
&=-\frac{i}{\mu}\pi_{\lambda, \mu, \nu}\left(\left( \tilde{X}_3-X_3-x_1 X_4\right)\kappa \right)h(u)\\
&=\frac{i}{\mu}\left(\pi_{\lambda, \mu, \nu}\left(X_4x_1\kappa\right)+\pi_{\lambda, \mu, \nu}\left(X_3\kappa \right)-\pi_{\lambda, \mu, \nu}\left(\tilde{X}_3 \kappa \right) \right)h(u)\,.
\end{align*}
Thus, since $\pi_{\lambda, \mu, \nu}(X_4)=i \lambda$, we finally get
\begin{equation}\label{pi(x2k)cartan}
\pi_{\lambda, \mu, \nu}(x_2\kappa)h(u)=\left[\left(-\frac{\lambda}{\mu}\right)\pi_{\lambda, \mu, \nu}(x_1 \kappa)+\frac{i}{\mu}( \pi_{\lambda, \mu, \nu}(X_3\kappa)-\pi_{\lambda, \mu, \nu}(\tilde{X}_3 \kappa)) \right]h(u)\,.
\end{equation}
Plugging \eqref{pi(x2k)cartan} into \eqref{dnu} yields
\begin{align*}
\frac{\partial}{\partial \nu} \{\pi_{\lambda, \mu, \nu}(\kappa)h(u)\}&=\frac{i\mu}{2(\lambda^2+\mu^2)}\pi_{\lambda, \mu, \nu}(x_1 \kappa)h(u)\nonumber\\
&-\frac{i \lambda}{2(\lambda^2+\mu^2)}\left[\left(-\frac{\lambda}{\mu}\right)\pi_{\lambda, \mu, \nu}(x_1\kappa)+\frac{i}{\mu}\left( \pi_{\lambda, \mu, \nu}(X_3\kappa)-\pi_{\lambda, \mu, \nu}(\tilde{X}_3 \kappa) \right) \right]h(u)\nonumber\\
&=\frac{i}{2\mu}\pi_{\lambda, \mu, \nu}(x_1 \kappa)h(u)\nonumber\\
&+\frac{\lambda}{2\mu(\lambda^2+\mu^2)}\left[\pi_{\lambda, \mu, \nu}(X_3\kappa)-\pi_{\lambda, \mu, \nu}(\tilde{X}_3 \kappa) \right]h(u)\,.
\end{align*}
Hence, the difference operator $\Delta_{x_1|\pi_{\lambda, \mu, \nu}}$ is acting on $\mathcal{S}(\mathbb{R})$ by
\begin{align}\label{dx1,cartan}
\pi_{\lambda, \mu, \nu}(x_1 \kappa)h(u)&=- 2 i\mu \frac{\partial}{\partial \nu}\{\pi_{\lambda, \mu, \nu}(\kappa)h(u)\}\nonumber\\
&+\frac{i \lambda}{(\lambda^2+\mu^2)}\left( \pi_{\lambda, \mu, \nu}(X_3)\pi_{\lambda, \mu, \nu}(\kappa)-\pi_{\lambda, \mu, \nu}(\kappa)\pi_{\lambda, \mu, \nu}(X_3) \right)h(u)\,.
\end{align}
Lastly, if we plug \eqref{dx1,cartan} into \eqref{pi(x2k)cartan}, we see that the difference operator $\Delta_{x_2|\pi_{\lambda, \mu, \nu}}$ is acting on $\mathcal{S}(\mathbb{R})$ via
\begin{eqnarray}\label{dx2,cartan}
\pi_{\lambda, \mu, \nu}(x_2 \kappa)h(u)&=& 2 i\lambda  \frac{\partial}{\partial \nu}\{\pi_{\lambda, \mu, \nu}(\kappa)h(u)\}\nonumber\\
&+&i\frac{\mu}{\lambda^2+\mu^2} \left(\pi_{\lambda,\mu,\nu}(X_3) \pi_{\lambda, \mu, \nu}(\kappa)-\pi_{\lambda, \mu, \nu}(\kappa)\pi_{\lambda, \mu, \nu}(X_3) \right)h(u)\,.
\end{eqnarray}
Now consider the case where $\mu=0$. Then $\lambda \neq 0$, and by \eqref{dnu} we get
\begin{equation}\label{casemu0}
\pi_{\lambda,\mu,\nu}(x_2 \kappa)h(u)=2i\lambda \frac{\partial}{\partial \nu} \{\pi_{\lambda, \mu, \nu}(\kappa)h(u)\}\,,
\end{equation}
and by the above observations 
\begin{align*}
\pi_{\lambda, \mu, \nu}(x_1 \kappa)h(u)&=\frac{1}{i \lambda}\pi_{\lambda, \mu, \nu}\left(X_4x_1\kappa\right)h(u)\\
&=-\frac{i}{\lambda}\pi_{\lambda, \mu, \nu}\left(\left( \tilde{X}_3-X_3-X_5x_2\right)\kappa \right)h(u)\\
&=\frac{i}{\lambda}\left(\pi_{\lambda, \mu, \nu}\left(X_5x_2\kappa\right)+\pi_{\lambda, \mu, \nu}\left(X_3\kappa \right)-\pi_{\lambda, \mu, \nu}\left(\tilde{X}_3 \kappa \right) \right)h(u)\,,
\end{align*}
while since $\pi_{\lambda,\mu,\nu}(X_5)=i\mu=0$ and by using \eqref{casemu0} we finally get 
\begin{equation}\label{casemu0x1}
\pi_{\lambda,\mu,\nu}(x_1 \kappa)h(u)=\frac{i}{\lambda}\left( \pi_{\lambda,\mu,\nu}(X_3)\pi_{\lambda,\mu,\nu}(\kappa)-\pi_{\lambda,\mu,\nu}(\kappa)\pi_{\lambda,\mu,\nu}(X_3) \right)h(u)\,,
\end{equation}
and this concludes the proof if one observes that \eqref{casemu0} and \eqref{casemu0x1} are particular cases of \eqref{dx2,cartan} and \eqref{dx1,cartan}, respectively. 
\end{proof}
\begin{prop}\label{dx3,car}
Let $\kappa \in \mathcal{K}_{a,b}(\mathcal{B}_5)$ be a tempered distribution, such that $x_3 \kappa$ is in some $\mathcal{K}_{a^{'},b^{'}}(\mathcal{B}_5)$. Then, the formula for the difference operator $\Delta_{x_3}$ acting on the space $\mathcal{S}(\mathbb{R})$ is the following: 
\begin{align*}
\Delta_{x_3} \pi_{\lambda, \mu, \nu}(\kappa)&=\frac{i}{\lambda} \left(\Delta_{x_2} \pi_{\lambda, \mu, \nu}(\kappa) \pi_{\lambda, \mu, \nu}(X_3)-\pi_{\lambda, \mu, \nu}(X_1) \pi_{\lambda, \mu, \nu}(\kappa)\right)\\
&+\frac{i}{\lambda}\pi_{\lambda, \mu, \nu}(\kappa) \pi_{\lambda, \mu, \nu}(X_1)+\frac{\mu}{\lambda}\Delta_{x_2}^{2}\pi_{\lambda, \mu,\nu}(\kappa)\,,
\end{align*}
where the difference operator $\Delta_{x_2 | \pi_{\lambda, \mu, \nu}}$, is as in Proposition \ref{dx1dx2,cartan}, and the operators $\pi_{\lambda, \mu, \nu}(X_i)$, $i=1,3,5$, are given in \eqref{inf2,car} and \eqref{inf3,4,5,car}.
\end{prop}
\begin{proof}
Since $X_4=\frac{\partial}{\partial x_4}$, then 
\begin{align*}
X_4x_3&=x_3X_4=x_3\frac{\partial}{\partial x_4}=x_3\frac{\partial}{\partial x_4}+\frac{\partial}{\partial x_1}-\frac{\partial}{\partial x_1}+x_2\frac{\partial}{\partial x_3}-x_2\frac{\partial}{\partial x_3}+\frac{x_{2}^{2}}{2}\frac{\partial}{\partial x_5}-\frac{x_{2}^{2}}{2}\frac{\partial}{\partial x_5}\\
&=X_1-x_2\tilde{X}_3+\frac{x_{2}^{2}}{2}X_5-\tilde{X_1}\,,
\end{align*}
where for the last equality we have used \eqref{left.inv.car},\eqref{dx345.right,cartan}, and \eqref{dx1,right,cartan}. Now, by \eqref{inf3,4,5,car} we get
\begin{align*}
\pi_{\lambda, \mu, \nu}(x_3 \kappa)&= \frac{1}{i \lambda} \pi_{\lambda, \mu, \nu}(X_4 x_3 \kappa)=\frac{1}{i \lambda}\pi_{\lambda, \mu, \nu} \left( \left(X_1-x_2 \tilde{X_3}+\frac{x_{2}^{2}}{2} X_5-\tilde{X_1} \right) \kappa \right)\\
&= \frac{1}{i \lambda} \left( \pi_{\lambda, \mu, \nu}(X_1 \kappa)-\pi_{\lambda, \mu, \nu}(\tilde{X_3}x_2 \kappa)+\pi_{\lambda, \mu, \nu}\left(X_5 \frac{x_{2}^{2}}{2} \kappa \right)-\pi_{\lambda, \mu, \nu}(\tilde{X_1} \kappa) \right)\\
&= \frac{1}{i \lambda} \left(\pi_{\lambda, \mu, \nu}(X_1) \pi_{\lambda, \mu, \nu}(\kappa)-\Delta_{x_2} \pi_{\lambda, \mu, \nu}(\kappa) \pi_{\lambda, \mu, \nu}(X_3)\right)\\
&+\frac{1}{i \lambda} \left(\frac{1}{2}\pi_{\lambda, \mu, \nu}(X_5) \Delta_{x_{2}}^{2}\pi_{\lambda, \mu, \nu}(\kappa)-\pi_{\lambda, \mu, \nu}(\kappa) \pi_{\lambda, \mu, \nu}(X_1) \right)\,,
\end{align*}
where $\pi_{\lambda,\mu,\nu}(X_5)=i\mu$, which shows the desired result.
\end{proof}

\begin{prop}\label{dx4,cartan}
Let $\kappa \in \mathcal{K}_{a,b}(\mathcal{B}_5)$, such that $x_4 \kappa$ is in some $\mathcal{K}_{a^{'},b^{'}}(\mathcal{B}_5)$. Then, we have for all $h \in \mathcal{S}(\mathbb{R})$ that 
\begin{align*}
(\Delta_{x_4} \pi_{\lambda, \mu, \nu} (\kappa))h(u)&= i \Bigg( \frac{\partial}{\partial \lambda} \left\lbrace \pi_{\lambda, \mu, \nu}(\kappa)h(u) \right\rbrace -\Bigg[\frac{2 \lambda \mu}{(\lambda^2+\mu^2)^2}\Bigg]  \Delta_{x_2} \pi_{\lambda, \mu, \nu}(\kappa)\left\lbrace \frac{d}{du}h \right\rbrace (u)\nonumber \\
&+\Bigg[\frac{\mu^2-\lambda^2}{(\lambda^2+\mu^2)^2}\Bigg] \Delta_{x_1} \pi_{\lambda, \mu, \nu}(\kappa)\left\lbrace\frac{d}{du}h\right\rbrace(u)\Bigg) \\
&+(b_1^1+b_1^2 u^2)\left\lbrace\Delta_{x_1} \pi_{\lambda, \mu, \nu} (\kappa)h(u) \right\rbrace +(b_{2}^{1}+ b_{2}^{2} u^2)\left\lbrace \Delta_{x_2}  \pi_{\lambda, \mu, \nu}(\kappa)h(u) \right\rbrace \nonumber\\
&+ u b_{1,2} \left\lbrace\Delta_{x_1} \Delta_{x_2} \pi_{\lambda, \mu, \nu}(\kappa)h(u) \right\rbrace+ b_{1,1,2} \left\lbrace\Delta_{x_1}^{2} \Delta_{x_2} \pi_{\lambda, \mu, \nu}(\kappa)h(u) \right\rbrace  \nonumber \\
&+ b_{1,2,2}\left\lbrace \Delta_{x_1} \Delta_{x_2}^{2} \pi_{\lambda, \mu, \nu}(\kappa)h(u) \right\rbrace b_{1,1,1}\left\lbrace\Delta_{{x_1}}^{3}\pi_{\lambda, \mu, \nu}(\kappa)h(u) \right\rbrace \nonumber \\
 &+ b_{2,2,2} \left\lbrace\Delta_{x_2}^{3} \pi_{\lambda, \mu, \nu}(\kappa)h(u) \right\rbrace \,,
\end{align*}
where the difference operators $\Delta_{x_i | \pi_{\lambda, \mu, \nu}}$, $i=1,2,3$, are as in Propositions \ref{dx1dx2,cartan} and \ref{dx3,car}, and 
\[
b_1^1=\frac{-\nu \lambda \mu }{(\lambda^2+\mu^2)^2}\,,\quad b_1^2=-\lambda \mu \,,\quad b_2^1=\frac{\nu \lambda^2-\nu \mu^2}{2(\lambda^2+\mu^2)^2}\,,
\]
\[
b_2^2=\frac{9 \lambda^6+3\mu^6+15 \lambda^2\mu^4+21\lambda^4\mu^2}{6(\lambda^2+\mu^2)^2}\,,\quad b_{1,2}=-2\lambda\,,\quad b_{1,1,2}=\frac{3\lambda^4+9\mu^4}{6(\lambda^2+\mu^2)^2}\,,
\]
\[
 b_{1,2,2}=\frac{-2\lambda \mu^3}{(\lambda^2+\mu^2)^2}\,,\quad b_{1,1,1}=\frac{4\lambda \mu^3}{3(\lambda^2+\mu^2)^2}\,,\quad b_{2,2,2}=\frac{2\mu^2\lambda^2-2\mu^4}{3(\lambda^2+\mu^2)^2}\,.
\]

\end{prop}
\begin{proof}
Let $\kappa$ be as in the statement. We start by differentiating with respect to $\lambda$ the group Fourier transform of $\kappa$, given in \eqref{group-fourier-transform-cartan}, at $h \in \mathcal{S}(\mathbb{R})$. This yields, 
\begin{align*}
\frac{\partial}{\partial \lambda} \{ \pi_{\lambda, \mu, \nu}(\kappa) h(u) \}&= \int_{\mathbb{R}^5} \kappa(x) \frac{\partial}{\partial \lambda} \left[ e^{iA^{*}} \cdot h \left(u- \frac{\lambda x_1+\mu x_2}{\lambda^2 + \mu^2} \right)\right] dx \\
&= \int_{\mathbb{R}_5} \kappa(x) \cdot e^{iA^{*}} \cdot \frac{\partial (iA^{*})}{\partial \lambda} \cdot h \left(u- \frac{\lambda x_1+\mu x_2}{\lambda^2 + \mu^2} \right)dx\\
&+ \int_{\mathbb{R}_5} \kappa(x) \cdot e^{iA^{*}}\\
& \left[ \frac{2 \lambda(\lambda x_1 + \mu x_2)-x_1 (\lambda^2 + \mu^2)}{(\lambda^2 + \mu^2)^{2}} \right]\frac{d}{du} h \left( u- \frac{\lambda x_1 + \mu x_2}{\lambda^2 +\mu^2} \right) dx\\
&=I+II\,. 
\end{align*}
Dealing with integral $II$, we see that 
\begin{align}\label{II}
II&= \int_{\mathbb{R}_5}\kappa(x)\cdot e^{iA^{*}} \left\lbrace x_2 \frac{2 \lambda \mu}{(\lambda^2 +\mu^2)^2}+x_1 \frac{- \mu^2+ \lambda^2}{(\lambda^2 + \mu^2)^2} \right\rbrace \frac{d}{du} h \left( u- \frac{\lambda x_1 + \mu x_2}{\lambda^2 +\mu^2} \right) dx \nonumber \\
&= \frac{2 \lambda \mu}{(\lambda^2 +\mu^2)^2} \int_{\mathcal{B}_5}x_2 \kappa(x)\cdot e^{iA^{*}} \cdot \frac{d}{du} h \left( u- \frac{\lambda x_1 + \mu x_2}{\lambda^2 +\mu^2} \right) dx \nonumber \\
&+ \frac{- \mu^2+ \lambda^2}{(\lambda^2 + \mu^2)^2}  \int_{\mathcal{B}_5}x_1 \kappa(x) \cdot e^{iA^{*}} \cdot\frac{d}{du} h \left( u- \frac{\lambda x_1 + \mu x_2}{\lambda^2 +\mu^2} \right) dx \nonumber \\
&= \Bigg( \Bigg[\frac{2 \lambda \mu}{(\lambda^2 +\mu^2)^2}\Bigg] \Delta_{x_2} \pi_{\lambda, \mu, \nu} (\kappa) \left\lbrace\frac{d}{du}h \right\rbrace+\Bigg[ \frac{- \mu^2+ \lambda^2}{(\lambda^2 + \mu^2)^2}\Bigg] \Delta_{x_1} \pi_{\lambda, \mu, \nu} (\kappa) \left\lbrace\frac{d}{du}h \right\rbrace \Bigg)(u)\,.
\end{align}
For the integral $I$, we see that since 
\begin{align}\label{II1}
\frac{\partial (iA^{*})}{\partial \lambda}&= i \left\lbrace -x_4 +(b_1^1 +b_1^2 u^2)x_1+(b_2^1+b_2^2 u^2)x_2+ u b_{1,2} x_1x_2+b_{1,1,2}x_1^2 x_2+b_{1,2,2}x_1 x_2^2 \right\rbrace\nonumber \\
&+ i \left\lbrace b_{1,1,1} x_1^3+b_{2,2,2} x_2^3 \right\rbrace\,,
\end{align}
then, reasoning as in the case of the integral $II$, we get
\begin{align*}
I&=i \Big[ \left\lbrace -\Delta_{x_4}\pi_{\lambda, \mu, \nu}(\kappa) +(b_1^1 +b_1^2 u^2)\Delta_{x_1}\pi_{\lambda, \mu, \nu}(\kappa)+(b_2^1+b_2^2 u^2)\Delta_{x_2}\pi_{\lambda, \mu, \nu}(\kappa)\right\rbrace h(u) \Big]\nonumber \\
&+ i\Big[ \left\lbrace u b_{1,2} \Delta_{x_1}\Delta_{x_2}\pi_{\lambda, \mu, \nu}(\kappa)+b_{1,1,2}\Delta_{x_1}^2 \Delta_{x_2}\pi_{\lambda, \mu, \nu}(\kappa)+b_{1,2,2}\Delta_{x_1}\Delta_{x_2}^{2}\pi_{\lambda, \mu, \nu}(\kappa) \right\rbrace h(u) \Big] \nonumber\\
&+i \Big[ \left\lbrace b_{1,1,1}\Delta_{x_1}^{3}\pi_{\lambda, \mu, \nu}(\kappa)+b_{2,2,2}\Delta_{x_2}^{3}\pi_{\lambda, \mu, \nu}(\kappa)\right\rbrace h(u) \Big]\,,
\end{align*}
where the $b$'s are as in the statement, and the proof is complete.
\end{proof}
\begin{prop}\label{dx5,cartan}
Let $\kappa \in \mathcal{K}_{a,b}(\mathcal{B}_5)$, such that $x_5 \kappa$ is in some $\mathcal{K}_{a^{'},b^{'}}(\mathcal{B}_5)$. Then, we have for all $h \in \mathcal{S}(\mathbb{R})$ that
\begin{align*}
(\Delta_{x_5} \pi_{\lambda, \mu, \nu} (\kappa))h(u)&= i \Bigg( \frac{\partial}{\partial \mu} \left\lbrace \pi_{\lambda, \mu, \nu}(\kappa)h(u) \right\rbrace-\Bigg[\frac{2 \lambda \mu}{(\lambda^2+\mu^2)^2}\Bigg]\Delta_{x_1} \pi_{\lambda, \mu, \nu}(\kappa)\left\lbrace \frac{d}{du}h \right\rbrace (u)\nonumber\\
&+ \Bigg[\frac{\lambda^2-\mu^2}{(\lambda^2+\mu^2)^2}\Bigg] \Delta_{x_2} \pi_{\lambda, \mu, \nu}(\kappa)\left\lbrace\frac{d}{du}h \right\rbrace (u)\Bigg) \\
&+(b_1^1+b_1^2 u^2)\left\lbrace\Delta_{x_1} \pi_{\lambda, \mu, \nu} (\kappa)h(u) \right\rbrace +(b_{2}^{1}+ b_{2}^{2} u^2)\left\lbrace \Delta_{x_2}  \pi_{\lambda, \mu, \nu}(\kappa)h(u)\right\rbrace\nonumber\\
&+ b_{1,1,2}\left\lbrace \Delta_{x_1}^{2} \Delta_{x_2} \pi_{\lambda, \mu, \nu}(\kappa)h(u) \right\rbrace++ b_{1,2,2} \left\lbrace\Delta_{x_1} \Delta_{x_2}^{2} \pi_{\lambda, \mu, \nu}(\kappa)h(u)\right\rbrace\nonumber \\
&+b_{1,1,1} \left\lbrace\Delta_{x_1}^{3} \pi_{\lambda, \mu, \nu}(\kappa)h(u)\right\rbrace+ b_{2,2,2}\left\lbrace \Delta_{x_2}^{3} \pi_{\lambda, \mu, \nu}(\kappa)h(u)\right\rbrace\,,
\end{align*}
where the difference operators $\Delta_{x_i | \pi_{\lambda, \mu, \nu}}$, $i=1,2,3$, are as in Propositions \ref{dx1dx2,cartan} and \ref{dx3,car}, and
\[
b_1^1=\frac{\nu \lambda^2-\nu \mu^2}{2(\lambda^2+\mu^2)^2}\,,\quad b_1^2=-\frac{3 \lambda^6+9\mu^6+15\lambda^4\mu^2+21\lambda^2\mu^4}{6(\lambda^2+\mu^2)^2}\,,\quad b_2^1=\frac{\nu \lambda \mu }{(\lambda^2+\mu^2)^2}\,,
\]
\[
b_2^2=\lambda \mu\,,\quad b_{1,1,2}=\frac{2\lambda^3\mu}{(\lambda^2+\mu^2)^2}\,,\quad b_{1,2,2}=\frac{\mu^4+3\mu^2\lambda^2}{(\lambda^2+\mu^2)^2}\,,
\]
\[
  b_{1,1,1}=\frac{2\lambda^4-2\lambda^2\mu^2}{3(\lambda^2+\mu^2)^2}\,,\quad b_{2,2,2}=\frac{-4\lambda^3\mu}{3(\lambda^2+\mu^2)^2}\,.
\]

\end{prop}
\begin{proof}
For $\kappa$ as in the statement, and for $h \in \mathcal{S}(\mathbb{R})$, we differentiate with respect to $\mu$ the expression \eqref{group-fourier-transform-cartan} at $h$. In particular, 
\begin{align*}
\frac{\partial}{\partial \mu} \{ \pi_{\lambda, \mu, \nu}(\kappa) h(u) \}&= \int_{\mathbb{R}^5} \kappa(x) \frac{\partial}{\partial \mu} \left[ e^{iA^{*}} \cdot h \left(u- \frac{\lambda x_1+\mu x_2}{\lambda^2 + \mu^2} \right)\right] dx \\
&= \int_{\mathbb{R}_5} \kappa(x) \cdot e^{iA^{*}} \cdot \frac{\partial (iA^{*})}{\partial \mu} \cdot h \left(u- \frac{\lambda x_1+\mu x_2}{\lambda^2 + \mu^2} \right)dx\\
&+ \int_{\mathbb{R}_5} \kappa(x) \cdot e^{iA^{*}}\\
& \left[ \frac{2 \mu(\lambda x_1 + \mu x_2)-x_2 (\lambda^2 + \mu^2)}{(\lambda^2 + \mu^2)^{2}} \right]\frac{d}{du} h \left( u- \frac{\lambda x_1 + \mu x_2}{\lambda^2 +\mu^2} \right) dx\\
&=I+II\,. 
\end{align*}
Dealing with integral $II$, we see that 
\begin{align}\label{II}
II&= \int_{\mathbb{R}_5}\kappa(x)\cdot e^{iA^{*}} \left\lbrace x_1 \frac{2 \lambda \mu}{(\lambda^2 +\mu^2)^2}+x_2 \frac{ \mu^2- \lambda^2}{(\lambda^2 + \mu^2)^2} \right\rbrace \frac{d}{du} h \left( u- \frac{\lambda x_1 + \mu x_2}{\lambda^2 +\mu^2} \right) dx \nonumber \\
&= \frac{2 \lambda \mu}{(\lambda^2 +\mu^2)^2} \int_{\mathcal{B}_5}x_1 \kappa(x) \cdot e^{iA^{*}} \cdot \frac{d}{du} h \left( u- \frac{\lambda x_1 + \mu x_2}{\lambda^2 +\mu^2} \right) dx \nonumber \\
&+ \frac{ \mu^2- \lambda^2}{(\lambda^2 + \mu^2)^2}  \int_{\mathcal{B}_5}x_2 \kappa(x) \cdot e^{iA^{*}} \cdot\frac{d}{du} h \left( u- \frac{\lambda x_1 + \mu x_2}{\lambda^2 +\mu^2} \right) dx \nonumber \\
&= \Bigg(\left[ \frac{2 \lambda \mu}{(\lambda^2 +\mu^2)^2}\right] \Delta_{x_1} \pi_{\lambda, \mu, \nu} (\kappa) \left\lbrace\frac{d}{du}h \right\rbrace + \left[\frac{ \mu^2- \lambda^2}{(\lambda^2 + \mu^2)^2}\right]  \Delta_{x_2} \pi_{\lambda, \mu, \nu} (\kappa)\left\lbrace \frac{d}{du}h \right\rbrace  \Bigg)(u)\,.
\end{align}
For the integral $I$, we see that since 
\begin{align}\label{II1}
\frac{\partial (iA^{*})}{\partial_{\mu}}&= i \left\lbrace -x_5 +(b_1^1 +b_1^2 u^2)x_1+(b_2^1+b_2^2 u^2)x_2+ b_{1,1,2}x_1^2 x_2\right\rbrace\nonumber \\
&+ i \left\lbrace b_{1,2,2}x_1 x_2^2 + b_{1,1,1} x_1^3+b_{2,2,2} x_2^3 \right\rbrace\,,
\end{align}
then, arguing as in the case of the integral $II$ in terms of difference operators, we get
\begin{align*}
I&=i \left[ \left\lbrace -\Delta_{x_5}\pi_{\lambda, \mu, \nu}(\kappa) +(b_1^1 +b_1^2 u^2)\Delta_{x_1}\pi_{\lambda, \mu, \nu}(\kappa)+(b_2^1+b_2^2 u^2)\Delta_{x_2}\pi_{\lambda, \mu, \nu}(\kappa)\right\rbrace h \right](u)\nonumber \\
&+ i \left[ \left\lbrace b_{1,1,2}\Delta_{x_1}^2 \Delta_{x_2}\pi_{\lambda, \mu, \nu}(\kappa)+ b_{1,2,2}\Delta_{x_1}\Delta_{x_2}^{2}\pi_{\lambda, \mu, \nu}(\kappa) \right\rbrace h \right](u)\nonumber\\
&+i \left[\left\lbrace b_{1,1,1}\Delta_{x_1}^{3}\pi_{\lambda, \mu, \nu}(\kappa)+b_{2,2,2}\Delta_{x_2}^{3}\pi_{\lambda, \mu, \nu}(\kappa)\right\rbrace h \right](u)\,.
\end{align*}
where the $b$'s are as in the statement, and the proof is complete.
\end{proof}

\section{$L^p-L^q$ mutlipliers on the Engel and the Cartan groups}
\label{SEC:multipliers}
The following theorem is a version of the Fourier multiplier theorem (see, \cite[Theorem 7.1]{AR17}), in the case of spectral multipliers on locally compact groups. In particular, the aforesaid theorem affirms that the $L^p-L^q$ norms of spectral multipliers $\phi \left( |\mathcal{L}| \right)$ depend on the growth rate of the spectral asymptotics of the considered operator $\mathcal{L}$.

By a left multiplier in this setting we mean the left-invariant operator that is $\tau$-measurable, where $\tau$ refers to the trace functional on the positive part
\[
VN_{R}^{+}(\G):=\{A \in VN_{R}(\G) : A^{*}=A>0\}\,,
\]
of the right group von Neumann algebra $VN_{R}(\G)$, see \cite{AR20} and \cite[Appendix B]{FR16} for a detailed definition of the above. 
\begin{thm}{\cite[Theorem 1.1]{AR17}}
\label{Theorem 1.1,AR}
Let $\G$ be a locally compact separable unimodular group and let $\mathcal{L}$ be a left Fourier multiplier on $\G$. Assume that $\phi$ is monotonically decreasing continuous function on $[0, \infty)$ such that 
\[
\phi(0)=1\,,
\]
\[
\lim_{s \rightarrow \infty} \phi(s)=0\,.
\]
Then we have the inequality 
\begin{equation*}
\label{Theorem 1.1,ineqiality}
\| \phi \left( | \mathcal{L}| \right)\|_{L^p(\G) \rightarrow L^q(\G)} \lesssim \sup_{s>0} \phi(s) \left[ \tau \left(E_{(0,s)}\left(|\mathcal{L}| \right) \right)\right]^{\frac{1}{p}-\frac{1}{q}}\,, \quad 1<p \leq 2 \leq q < \infty\,.
\end{equation*}
\end{thm}
 We denote by $E_{(0,s)}\left(|\mathcal{L}| \right)$, the spectral projection associated to the operator $\mathcal{L}$ to the interval $(0,s)$. Notice also that since $E_{(0,s)}\left(|\mathcal{L}| \right) \in VN_{R}(\G)$, it makes sense to write $\tau \left( E_{(0,s)}\left( | \mathcal{L} | \right) \right)$, see \cite[Section 2]{AR20} for precise definitions and explanations.
 
To motivate the examples that we give later, let us first recall some well-known results in the compact case and in the case of the Heisenberg groups $\mathbb{H}^n$. Let $-\Delta_{sub}$ be the positive sub-Laplacian on a compact Lie group $\G$, or in the graded case, on the Heisenberg group $\mathbb{H}^n$. Then, the trace of the spectral projections has the asymptotics
\begin{equation}
\label{compact,heisen.}
\tau \left(E_{(0,s)}\left(-\Delta_{sub} \right) \right)\lesssim s^{Q/2}\,,\quad \textrm{as}\, \quad s \rightarrow \infty\,,
\end{equation}
where $Q$ is the homogeneous dimension in the case of $\mathbb{H}^{n}$, and the Hausdorff dimension generated by the control distanse of $\Delta_{sub}$ in the case of the compact group $\G$, respectively. This result has been shown in the compact case in \cite{HK16}, and in the Heisenberg setting in \cite{AR20}. Thus, an application of Theorem \ref{Theorem 1.1,AR} yields for both cases:
\begin{itemize}
\item $\| \phi \left( -\Delta_{sub} \right)\|_{L^p(\G) \rightarrow L^q(\G)} \lesssim \sup_{s>0} \phi(s)\cdot s^{Q/2}$, for $\phi$ as in Theorem \ref{Theorem 1.1,AR}, provided that $1<p \leq 2 \leq q< \infty$.
\item For $f$ in the Schwartz space of each group, one has
\begin{equation}
\label{sob,estimates,gen}
\|f\|_{W^{b,q}_{\Delta_{sub}}} \leq C \|f\|_{W^{a,p}_{\Delta_{sub}}}\,,\quad \textrm{if} \quad a-b \leq Q \left(\frac{1}{p}-\frac{1}{q} \right)\,,
\end{equation}
where the last inequality is the well-known Sobolev embedding estimates. Recall the definition of the Sobolev norm associated to the sub-Laplacian, that is,
\[
\| f \|_{W^{a,p}_{\Delta_{sub}}}:=\|(I-\Delta_{sub})^{a/2} f \|_{L^p(\G)}\,.
\]
\end{itemize}
Recently, in \cite{RR18}, the authors, employing different machinery from the one we have used in this work, proved that \eqref{compact,heisen.} holds true for any graded Lie group, and thus, recover the well-established, frst for the case where the Sobolev norm is associated with the sub-Laplacian, see \cite{Fol75}, and later on for the one associated with any Rockland operator, see \cite{FR16}, Sobolev embedding estimates of the form \eqref{sob,estimates,gen}.

In this paper, we prove the $L^p-L^q$ boundedness for the spectral multiplier of a left-invariant, non-Rockland operator on each group. Notably, we extend the Sobolev type estimates, and estimate the time decay of the solution of the heat equations in each setting.
\begin{exa}($L^p-L^q$ multipliers on the  Engel group).
\label{multipliers,engel}
Let $\phi$ be as in Theorem \ref{Theorem 1.1,AR}.
Then, for
\[
A=- \left(X_1^2+X_2^2+X_3^2+X_4^2+X_4^{-2} \right)\,,
\]
with $X_i$, $i=1, \cdots,4$, being the left-invariant vector fields on $\mathcal{B}_4$, given in \eqref{left.inv.eng}, we have the inequality 
\[
\| \phi(A) \|_{L^{p}(\mathcal{B}_4) \rightarrow L^{q}(\mathcal{B}_4)} \lesssim C \cdot\sup_{s>0}\phi(s)\cdot s^{\frac{3}{r}}\,,
\]
where $\frac{1}{r}=\frac{1}{p}-\frac{1}{q}$, for $1<p \leq 2 \leq q < \infty$. In addition, for $f \in \mathcal{S}(\mathcal{B}_4)$, the Sobolev-type inequalities
\begin{equation*}
\label{for_phi,engel}
\|(I+A)^{b} f\|_{L^{q}(\mathcal{B}_4)} \leq \| (I+A)^{a}f \|_{L^{p}(\mathcal{B}_4)}
\end{equation*}
hold true for $a, b \in\mathbb{R}$, such that $a-b\geq 	\frac{3}{r}$, and for $p,q$ in the above range.
\end{exa}Before turning on to prove the above let us first note that the left-invariant Fourier multiplier $X_{4}^{-2}$, if $\lambda \neq 0$, can be realised through its symbol, i.e.,
\[
X_{4}^{-2}f=\mathcal{F}_{\mathcal{B}_4}(\pi_{\lambda,\mu}(X_{4})^{-2}\pi_{\lambda,\mu}(f))=\mathcal{F}_{\mathcal{B}_4}((i\lambda)^{-2}\pi_{\lambda,\mu}(f))\,,\quad f \in \mathcal{S}(\mathcal{B}_4)\,.
\]

 \begin{proof}[Proof of Example \ref{multipliers,engel}]
 For the operator $A$ as in the statement, and by \cite[Theorem 1 on page 225]{Dix81} we get
 \begin{equation}
 \label{Dix,trace,eng}
 \tau \left( E_{(0,s)}(A)\right)\sim\int_{\hat{\mathcal{B}_4}} \tau \left(E_{(0,s)}\left[ \pi_{\lambda,\mu}(A) \right]  \right)    d\lambda\, d\mu\,.
 \end{equation}
 For $\mu \in \mathbb{R}\,, \lambda \neq 0$, the operator $\pi_{\lambda, \mu}(A): \mathcal{S}(\mathbb{R}) \rightarrow \mathcal{S}(\mathbb{R})$, is the symbol of $A$, given by  
 \begin{equation*}\label{R,eng}
 \pi_{\lambda, \mu}(A)=-\frac{d^2}{du^2}+ \frac{1}{4}\left( \lambda u^2- \frac{\mu}{\lambda} \right)^{2}+(\lambda u)^{2}+\lambda^{2}+\lambda^{-2}\,,\quad u \in \mathbb{R}\,,
 \end{equation*}
 as \eqref{inf1,2,eng} and \eqref{inf3,4,5,eng} imply. Observe that $\pi_{\lambda, \mu}(A)$ is of the form 
 \begin{equation*}
 \label{anharmonic}
 -\frac{d^2}{du^2}+V(u)\,,
 \end{equation*}
 i.e. is a particular case of the anharmonic oscillator on $\mathcal{S}(\mathbb{R}) \subset L^2(\mathbb{R})$. Thus, it is known, see \cite{CDR18}, that its spectrum is discrete, and if $N_{A}(s)$ denotes the number of its eigenvalues that are less that $s$, then 
 \begin{equation}\label{type1,eng}
 \tau \left(E_{(0,s)}\left[ \pi_{\lambda,\mu}(A) \right] \right) = N_{A}(s)\,,
 \end{equation}
 since $\mathcal{B}_4$ is of type $I$, see \cite{AR20}. Let $\sigma_{A}$ be the Weyl symbol of $A$. Then, for $\mu \in \mathbb{R}$, and for $\lambda>0$, without loss of generality,
 \[
 \sigma_{A}(u,\xi)= \xi^{2}+\frac{1}{4}\left( \lambda u^2- \frac{\mu}{\lambda} \right)^{2}+(\lambda u)^{2}+\lambda^{2}+\lambda^{-2}\,.
 \]
Then, \cite[Theorem 3.2]{BBR96} implies that 
   \begin{equation}
   \label{BBR, engel}
   N_{A}(s) \lesssim C \int_{ \sigma_{A}(u,\xi)<s} 1 \cdot du\,d\xi\,, \quad \textrm{as}\quad s\rightarrow \infty\,.
   \end{equation}
   Assuming that $\sigma_{A}(u,\xi)<s$, it follows immediately that
   \begin{equation}
   \label{for_lamda,xi,engel}
   s^{-1/2}< \lambda<s^{1/2}\,,\quad |\xi|<s^{1/2}\,.
   \end{equation}
   and
   \begin{equation}
   \label{sigma<s,engel}
   \frac{1}{4}\left( \lambda u^2-\frac{\mu}{\lambda}\right)^2, (\lambda u)^2<s\,.
   \end{equation}
   Now, by \eqref{for_lamda,xi,engel} and \eqref{sigma<s,engel}, we get 
   $ |u|< \frac{s^{1/2}}{\lambda}< \frac{s^{1/2}}{s^{-1/2}}=s$,
   and
   \begin{equation*}
-2s^{1/2} \cdot s^{1/2}<-2\lambda s^{1/2}<2(\lambda^2 u^2- \lambda s^{1/2})<\mu< 2(\lambda^{2}u^{2}+\lambda s^{1/2})<2(s+s^{1/2}\cdot s^{1/2})\,.
\end{equation*}
In particular, we have $-2s<\mu<4s$. Then, by \eqref{BBR, engel}, and since $|u|<s, |\xi|<s^{1/2}$, we have $N_{A}(s) \lesssim C \cdot s^{3/2}$, for $s \rightarrow \infty$, whereas by \eqref{type1,eng}, the integration on the dual $\hat{\mathcal{B}_4}$ in \eqref{Dix,trace,eng} for $\lambda \in (s^{-1/2},s^{1/2})$, $\mu \in (-2s,4s)$ becomes
   \begin{align*}
    \tau \left( E_{(0,s)}(A)\right)&\sim\int_{s^{-1/2}}^{s^{1/2}} \int_{-2s}^{4s} \tau \left(E_{(0,s)}\left[ \pi_{\lambda,\mu}(A) \right]  \right)    d\lambda\, d\mu\\
    & \lesssim C \cdot \int_{s^{-1/2}}^{s^{1/2}} \int_{-2s}^{4s} s^{3/2} d\lambda\,d\mu \lesssim C^{'}\cdot s^3\,,\quad s \rightarrow \infty\,.
   \end{align*}
   Summarising we have $ \tau \left( E_{(0,s)}(A)\right) \lesssim C^{'} \cdot s^3$, for $s \rightarrow \infty$. Then, for $\phi$ as in the statement, an application of Theorem \ref{Theorem 1.1,AR}, yields
   \[
   \|\phi\left( A \right)\|_{L^{p}(\mathcal{B}_4) \rightarrow L^{q}(\mathcal{B}_4)} \lesssim C^{'} \cdot \sup_{s>0} \phi(s) \cdot  \tau \left( E_{(0,s)}(A)\right)^{\frac{1}{p}-\frac{1}{q}} \lesssim C^{'} \cdot \sup_{s>0}\phi(s) \cdot  s^{\frac{3}{r}}\,,
   \]
   where $\frac{1}{r}=\frac{1}{p}-\frac{1}{q}$, and $1<p \leq 2 \leq q < \infty$, and this shows the first claim in Example \ref{multipliers,engel}.  Now, if we choose 
\begin{equation}
\label{def_of_phi,eng}   
   \phi=\phi(s):=\frac{1}{(1+s)^{a-b}}\,,\quad s>0\,,
   \end{equation}
   then, for $f \in \mathcal{S}(\mathcal{B}_4)$,
   \[
 \|(I+A)^{-a} (I+A)^{b} f\|_{L^{q}(\mathcal{B}_4)} \lesssim \|f \|_{L^{p}(\mathcal{B}_4)}\,,
   \]
   or, as claimed in the statement,
   \[
   \|(I+A)^{b} f\|_{L^{q}(\mathcal{B}_4)} \lesssim \| (I+A)^{a}f \|_{L^{p}(\mathcal{B}_4)}\,.
   \]
   The proof is complete.
 \end{proof}
 We note that, in the above example, the operator $A$ can be generalised as being any non-Rockland operator of the form
\[
A=-(X_{1}^{2}+ X_{2}^{2n_2}\pm X_{3}^{2n_3}\pm X_{4}^{2n_3}\pm X_{4}^{-2n_4})\,,
\]
where $n_i \geq 2$, $n_i \in \mathbb{N}$ for $i=1,\cdots,4$, and for $i=3,4$ we choose the positive sign of the corresponding left-invariant operator if $n_i$ is of the form $n_i=1+2k$, $k \geq 1$, and the negative sign, otherwise. 
 \begin{rem}(Sobolev-type estimates on the Engel group)
 \label{Sob,emb2,eng}
 For $b=0$ in the definition of the function $\phi$ in \eqref{def_of_phi,eng}, we obtain Sobolev-type embedding inequalities of the form
 \[
 \|f\|_{L^{q}(\mathcal{B}_4)} \leq C \| (I+A)^a f\|_{L^{p}(\mathcal{B}_4)}\,,
 \]
 for $a \geq \frac{3}{r}$ and $1<p \leq 2 \leq q < \infty$.
 \end{rem}
 \begin{rem}(The $A$-heat equation)
 \label{heat,eng}
 For the operator $A$ as in Example \eqref{multipliers,engel} consider the heat equation 
 \begin{equation}
 \label{heat.eq,eng}
 \partial_t u+A u=0\,,\quad u(0)=u_0\,.
 \end{equation}
It is not difficult to check that, for each $t>0$, the function $u(t)=u(t,x):=e^{-tA}u_0$, is a solution of the initial value problem \eqref{heat.eq,eng}. For $t>0$, also set $\phi=\phi(s):=e^{-ts}$. Then $\phi$ satisfies the assumptions stated before in Theorem \ref{Theorem 1.1,AR}, and as an application of the Example \ref{multipliers,engel} we get
 \begin{equation}\label{norm.op.heat.eng}
 \| e^{-tA}\|_{L^p(\mathcal{B}_4) \rightarrow L^q(\mathcal{B}_4)} \lesssim C \cdot \sup_{s>0}e^{-ts} \cdot s^{\frac{3}{r}}\,,
 \end{equation}
 where $\frac{1}{r}=\frac{1}{p}-\frac{1}{q}$, and $1<p \leq 2 \leq q < \infty$. Using techniques of standard mathematical analysis, we see that 
 \[
 \sup_{s>0} e^{-ts} \cdot s^{\frac{3}{r}}=\left( \frac{3}{tr} \right)^{\frac{3}{r}}e^{-\frac{3}{r}}:=C_{p,q} t^{-3\left(\frac{1}{p}-\frac{1}{q}\right)}\,.
 \]
 Indeed, if we consider the function $g(s)= e^{-ts} \cdot s^{\frac{3}{r}}$, then $g^{'}(s)=s^{\frac{3}{r}-1}e^{-ts} \left( \frac{3}{r}-st \right)$, whereas $g^{'}\left( \frac{3}{rt} \right)=0$, and $g^{''}\left( \frac{3}{rt} \right)<0$. Then, by \eqref{norm.op.heat.eng}, we obtain
 \[
 \|u(t, \cdot) \|_{L^q(\mathcal{B}_4)} \lesssim C_{p,q} t^{-3\left(\frac{1}{p}-\frac{1}{q}\right)} \|u_0\|_{L^{p}(\mathcal{B}_4)}\,, 1<p \leq 2 \leq q< \infty\,,
 \]
 where the last equations yields the time decay for the solution of the heat equation in this setting. 
 \end{rem}
 
 \begin{exa}
 \label{multipliers,cartan}($L^p-L^q$ multipliers on the Cartan group).
 Let $\phi$ be a function as in Theorem \ref{Theorem 1.1,AR}. Let
 \[
 B=-\left(X_1^2+X_2^2+X_3^2+X_4^2+X_5^2+X_4^{-2}+X_5^{-2} \right)\,,
 \]
 where $X_i$ are the left-invariant vector fields on $\mathcal{B}_5$, given in \eqref{left.inv.car}. Then,
 \[
 \| \phi(B) \|_{L^p(\mathcal{B}_5) \rightarrow L^q(\mathcal{B}_5)} \lesssim C \cdot \sup_{s>0} \phi(s) \cdot s^{\frac{9}{2r}}\,,
 \]
 where $\frac{1}{r}=\frac{1}{p}-\frac{1}{q}$, for $1<p \leq 2 \leq q < \infty$. In addition, for $f \in \mathcal{S}(\mathcal{B}_5)$, we get the Sobolev-type inequalities of the form 
\begin{equation}
\label{Sob,cartan}
\| (I+B)^b f \|_{L^q(\mathcal{B}_5)} \leq \| (I+B)^a f \|_{L^q(\mathcal{B}_5)}\,,
\end{equation}
for $a,b \in \mathbb{R}$, such that $a-b\geq \frac{9}{2r}$ and for $p,q$ in the range above.
\end{exa}
As before, the left-invariant Fourier multiplier $X_{5}^{-2}$, if $\mu\neq 0$, can be realised as
\[
X_{5}^{-2}f=\mathcal{F}_{\mathcal{B}_5}(\pi_{\lambda,\mu,\nu}(X_5)^{-2}\pi_{\lambda,\mu,\nu}(f))=\mathcal{F}_{\mathcal{B}_4}((i\mu)^{-2}\pi_{\lambda,\mu,\nu}(f))\,,\quad f \in \mathcal{S}(\mathcal{B}_5)\,,
\]
and similarly for $X_{4}^{-2}$.
\begin{proof}[Proof of Example \ref{multipliers,cartan}]
Let the operator $B$ be as in the statement. Then, an application of \cite[Theorem 1 on page 225]{Dix81} yields
\begin{equation*}
 \tau \left( E_{(0,s)}(B)\right)\sim\int_{\hat{\mathcal{B}_5}} \tau \left(E_{(0,s)}\left[ \pi_{\lambda,\mu, \nu}(B) \right]  \right)    d\lambda\, d\mu\,d\nu\,,
\end{equation*} 
or, for our purposes,
\begin{equation}
\label{Dix, cartan}
 \tau \left( E_{(0,s)}(B)\right)\sim\int_{\hat{\mathcal{B}_5}\setminus \{\lambda=0\}\cup \{\mu=0\}}\tau \left(E_{(0,s)}\left[ \pi_{\lambda,\mu, \nu}(B) \right]  \right)    d\lambda\, d\mu\,d\nu\,,
\end{equation} 
where $d\lambda\,d\mu\,d\nu$ is the Plancherel measure on $\hat{\mathcal{B}_5}$ found in \cite{Dix57}. For $\nu \in \mathbb{R}\,, \lambda,\mu \neq 0$, and $u \in \mathbb{R}$ the operator $\pi_{\lambda, \mu, \nu}(B) : \mathcal{S}(\mathbb{R}) \subset L^2(\mathbb{R}) \rightarrow L^2(\mathbb{R})$, is the symbol of $B$ given by 
\[
\pi_{\lambda, \mu, \nu}(B)=-\frac{1}{\lambda^2+\mu^2} \frac{d^2}{du^2}+\frac{1}{4}\frac{(\nu+(\lambda^2+\mu^2)^{2}u^2)^2}{\lambda^2+\mu^2}+(\lambda^2+\mu^2)^2 u^2+\lambda^{2}+\mu^{2}+\lambda^{-2}+\mu^{-2}\,,
\]
as \eqref{inf2,car} and \eqref{inf3,4,5,car} give. Arguing as we did for the proof of the Example \ref{multipliers,engel}, we have
\begin{equation}\label{type1,car}
 \tau \left(E_{(0,s)}\left[ \pi_{\lambda,\mu,\nu}(B) \right] \right) = N_{B}(s)\,,
 \end{equation}
 where $N_{B}(s)$ denotes the number of the eigenvalues of $\pi_{\lambda, \mu, \nu}(B)$ that are less than $s>0$.
 Observe that $\pi_{\lambda, \mu, \nu}(B)$ is a rescaled anharmonic oscillator. Now, if we consider the operator $\pi_{\lambda, \mu, \nu}(B^{'})=(\lambda^2+\mu^2) \cdot\pi_{\lambda, \mu, \nu}(B) $ on $L^2(\mathbb{R})$, then by \cite[Theorem 3.2]{BBR96}, for $s \rightarrow \infty$, we have
\begin{equation}
\label{BBG,cartan}
N_{B}(s)=N_{B^{'}}((\lambda^2+\mu^2)\cdot s) \lesssim C \cdot \int_{\sigma_{B^{'}}(u,\xi)< (\lambda^2+\mu^2)\cdot s} 1 \cdot du\,d\xi\,,
\end{equation}
where with $N_{B^{'}}((\lambda^2 + \mu^2)\cdot s)$ we have denoted the number of eigenvalues of $\pi_{\lambda, \mu, \nu}(B^{'})$, that are less than $(\lambda^2+\mu^2) \cdot s$, and with $\sigma_{B^{'}}$, the Weyl symbol of the anharmonic oscillator $\pi_{\lambda, \mu, \nu}(B^{'})$. Then,
\[
\sigma_{B^{'}}(u,\xi)=\xi^2+ (\nu+(\lambda^2+\mu^2)^{2}u^2)^{2}/4+(\lambda^2+\mu^2)^3 u^2 + (\lambda^2+\mu^2) (\lambda^{2}+\mu^{2})+ (\lambda^2+\mu^2) (\lambda^{-2}+\mu^{-2})\,,
\]
and without loss of generality we can assume $\lambda, \mu>0$. Now, if $\sigma_{B^{'}}(u,\xi)<(\lambda^2+\mu^2) \cdot s$, then 
\begin{equation}
\label{for_l,m,cartan}
s^{-1/2}<\lambda, \mu< s^{1/2}\,, \quad \textrm{and}
\end{equation}
\begin{equation}
\label{for_x,xi,nu,cartan}
(\lambda^2+\mu^2)^{2}u^2<s\,,\quad (\nu+(\lambda^2+\mu^2)^{2}u^2)^{2}<4 s (\lambda^2+\mu^2)\,,\quad \xi^2 < s(\lambda^2+\mu^2)\,.
\end{equation}
Now, by \eqref{for_l,m,cartan} and \eqref{for_x,xi,nu,cartan}, we get $|u|<\frac{s^{1/2}}{\lambda^2+\mu^2}<\frac{s^{1/2}}{s^{-1}+s^{-1}}=\frac{s^{3/2}}{2}$, $| \xi |<\sqrt{2} s^{1/2}\cdot s^{1/2}$ and
\[
\nu^{2}<(\nu+(\lambda^{2}+\mu^{2})^{2}u^{2})^{2}<4s (\lambda^{2}+\mu^{2})<4s^{2}\,.
\]
In particular, we have $| \nu |<2s$. Then, by \eqref{BBG,cartan}, and for $u,\xi$ in the above range, we have 
\begin{eqnarray}
\label{N^',cartan}
N_{B}(s)=N_{B^{'}}((\lambda^2+\mu^2) \cdot s) \lesssim C \cdot s^{3/2} \cdot 2\sqrt{2} s \lesssim 2\sqrt{2} C \cdot s^{5/2}\,,\quad s \rightarrow \infty\,.
\end{eqnarray}
Collecting, \eqref{type1,car} and \eqref{N^',cartan}, the integration on the dual $\hat{\mathcal{B}_5}$ that appears in \eqref{Dix, cartan}, for $\nu \in (-2s,2s)$ and $\lambda, \mu \in (s^{-1/2},s^{1/2})$, becomes
\begin{align*}
\tau \left(E_{(0,s)} \left( B\right) \right)&\sim\int_{-2s}^{2s} \int_{s^{-1/2}}^{s^{1/2}} \int_{s^{-1/2}}^{s^{1/2}}  \tau \left(E_{(0,s)}\left[ \pi_{\lambda,\mu, \nu}(B) \right]  \right)    d\lambda\, d\mu\,d\nu\\
&=\int_{-2s}^{2s} \int_{s^{-1/2}}^{s^{1/2}} \int_{s^{-1/2}}^{s^{1/2}} N_{\mathcal{B}_5}(s) d\lambda\, d\mu\,d\nu\\
 &\lesssim 2\sqrt{2} C \cdot \int_{-2s}^{2s} \int_{s^{-1/2}}^{s^{1/2}} \int_{s^{-1/2}}^{s^{1/2}} s^{5/2}  d\lambda\, d\mu\,d\nu \lesssim C^{'} \cdot s^{9/2}\,,\quad s \rightarrow \infty\,.
\end{align*}
 Then, for $\phi$ as in the statement, an application of Theorem \ref{Theorem 1.1,AR}, yields
   \begin{equation}
   \label{phi,ineq,cartan}
   \|\phi\left( B \right)\|_{L^{p}(\mathcal{B}_5) \rightarrow L^{q}(\mathcal{B}_5)} \lesssim C^{'} \cdot \sup_{s>0} \phi(s) \cdot  \tau \left( E_{(0,s)}(B)\right)^{\frac{1}{p}-\frac{1}{q}} \lesssim C^{'} \cdot  s^{\frac{9}{2r}}\,,
   \end{equation}
   where $\frac{1}{r}=\frac{1}{p}-\frac{1}{q}$, given that $1<p \leq 2 \leq q < \infty$. Finally, plugging $\phi$ as in \eqref{def_of_phi,eng} into \eqref{phi,ineq,cartan}, we get the Sobolev-type estimates  \eqref{Sob,cartan}.
\end{proof}
We note that, in the above example, the operator $B$ cn be generalised as being any non-Rockland operator of the form 
\[
B=-(X_{1}^{2}+X_{2}^{2n_2}\pm X_{3}^{2n_3} \pm X_{4}^{2n_4} \pm X_{5}^{2n_5} \pm X_{4}^{2n_6} \pm X_{5}^{2n_7})\,,
\]
where $n_1 \geq 2$, $n_1 \in \mathbb{N}$ for $i=1,\cdots,7$ and for $i=3,\cdots,7$ we choose the positive sign of the corresponding left-invariant operator if $n_i$ is of the form $n_i=1+2k$, $k \geq 1$, and the negative sign, otherwise. 
 \begin{rem}
 \label{Sob,emb2,car}
 For $b=0$ in the definition of the function $\phi$ in \eqref{def_of_phi,eng}, we obtain Sobolev-type embedding inequalities of the form
 \[
 \|f\|_{L^{q}(\mathcal{B}_5)} \leq C \| (I+B)^a f\|_{L^{p}(\mathcal{B}_5)}\,,
 \]
 for $a \geq \frac{9}{2r}$ and $1<p \leq 2 \leq q < \infty$.
 \end{rem}
 \begin{rem}(The $B$-heat equation)
 \label{heat.car} For the operator $B$ as in Example \ref{multipliers,cartan} consider the heat equation given by
 \begin{equation}
 \label{heat.eq.car}
 \partial_t u +Bu=0\,,\quad u(0)=u_0\,.
 \end{equation}
 Similarly to what has been done in Remark \ref{heat,eng}, we consider for $t>0$ the function $u(t)=u(t,x):=e^{-tB}u_0$ to be the solution of the initial value problem \eqref{heat.eq.car}, as well as the function $\phi=\phi(s):=e^{-ts}$ satisfying the assumptions given in Theorem \ref{Theorem 1.1,AR}. Reasoning as we did before in Remark \ref{heat,eng}, an application of Example \ref{multipliers,cartan} provides us with an estimate for time delay of the hear kernel of \eqref{heat.eq.car}, i.e.,we have
 \[
 \|u(t, \cdot) \|_{L^q(\mathcal{B}_5)} \lesssim C^{'}_{p,q}t^{-\frac{9}{2}\left( \frac{1}{p}-\frac{1}{q}\right)}\|u_0\|_{L^p(\mathcal{B}_5)}\,,1<p \leq 2 \leq q< \infty\,.
 \]
 \end{rem}
 \begin{rem}
It is interesting to mention that, in the above Examples, we proceed with Theorem \ref{Theorem 1.1,AR}, without knowing the explicit formulas for the eigenvalues of the left Fourier multiplier considered in each group. Indeed, in \cite[Example 9.5]{AR20}, the proof of the above statements in the setting of the Heisenberg group $\mathbb{H}^n$, relies on the fact that the symbol of the canonical sub-Laplacian on $\mathbb{H}^n$, is the harmonic oscillator, and thus the eigenvalues are the well-known Hermite functions.
 \end{rem}

\appendix
\section{Plancherel measure on the dual of the Engel group}\label{sec.foo} 

The first part of the abstract Plancherel theorem on unimodular, type I Lie groups, obtained by Dixmier in (\cite[Section 18.8]{Dix77}), is the Plancherel formula, which ensures the existence of a unique positive $\sigma$-finite measure $m=m(\pi)$, $\pi \in \hat{\G}$ on $\G$, called the Plancherel measure, such that for any $ f \in C_{c}(\G)$ we have
\begin{equation*}
\label{Plancherel-gen}
\int_{\G} |f(x) |^2 dx= \int_{\hat{\G}}\| \hat{f}(\pi) \|^{2}_{\textrm{HS}(\mathcal{H}_{\pi})}dm(\pi)\,,
\end{equation*}
where $\hat{f}(\pi)$ is an endomorphism on $\mathcal{H}_{\pi}$. The above formula shows that the group Fourier transform is a Hilbert-Schmidt operator, and in particular that the group Fourier transform is an isometry from the space of smooth functions with compact support $C_{c}(\G)$ endowed with the $L^2(\G)$ norm, to the Hilbert space
\[
L^{2}(\hat{\G)}:= \int_{\hat{G}}^{\oplus} \textrm{HS}(\mathcal{H}_{\pi})dm(\pi)\,,
\]
introduced by Dixmier in \cite[Part ii ch. I]{Dix81}, of fields of Hilbert-Schmidt operators that are square integrable in the above sense. 

  In the next Proposition we obtain a concrete formula for the Plancherel formula in the setting of the Engel group, and this implies specifying the Plancherel measure on the dual of the group. As mentioned earlier in Section \ref{SEC:multipliers}, the Plancherel measure in this case has been calculated by Dixmier in \cite{Dix57} using other methods. Here we give a straightforward proof.
  \begin{prop}[Plancherel formula in $\mathcal{B}_4$]
  \label{Plansheremeasure}
Let $f\in S(\mathcal{B}_{4})$. Then for each $\lambda \in \mathbb{R}\setminus\{0\},\mu\in\mathbb{R}$, the operator $\hat{f}(\pi_{\lambda, \mu})$ acting on $L^2(\mathbb{R})$ is the Hilbert-Schmidt operator with integral kernel $\mathcal{K}_{f,\lambda,\mu} : \mathbb{R} \times \mathbb{R} \longrightarrow \mathbb{C}$,
given by
\[
\mathcal{K}_{f,\lambda,\mu}(u,v)=(2\pi) \mathcal{F}_{\mathbb{R}^3}(f)(u-v,\frac{\lambda}{2}v^2-\frac{\mu}{2\lambda},-\lambda v,\lambda)\,,
\]
and Hilbert-Schmidt norm
\[
\Arrowvert \hat{f}(\pi_{\lambda, \mu})\Arrowvert_{\textrm{HS}(L^2(\mathbb{R}))}= (2\pi^{\frac{3}{2}}) \int_{\mathbb{R}} \int_{\mathbb{R}}\vert \mathcal{F}_{\mathbb{R}^3}(f)(u-v, \frac{\lambda}{2}v^2-\frac{\mu}{2\lambda}, -\lambda v, \lambda) \vert^2 \,du\,dv.
\]
In addition, we have
\[
\int_{\mathcal{B}_4} \vert f(x_1,x_2,x_3,x_4) \vert^{2} \,dx_1\,dx_2\,dx_3\,dx_4= \pi^{\frac{3}{2}} \cdot \int_{\lambda \in \mathbb{R} \setminus \{0\}} \int_{\mu \in \mathbb{R}} \Arrowvert \hat{f}(\pi_{\lambda, \mu})\Arrowvert^{2}_{HS(L^2(\mathbb{R}))}\,d\mu\,d\lambda\,.
\]
Furthermore, we have that the Plancherel measure on $\hat{\mathcal{B}}_4$ is the Lebesgue measure on $\mathbb{R}^2$, namely
	$
	dm(\pi_{\lambda, \mu})\equiv \pi\cdot \,\,d\lambda\,d\mu\,.
	$
	
\end{prop}

Before giving the proof, let us notice that the Schwartz space of $\mathcal{B}_4$ is in fact $S(\mathbb{R}^4)$, see \cite[Section 3.1.9]{FR16}.
\begin{proof}
	By \eqref{fourier,tansform,engel,pseudo-dif}, we have for $h \in L^{2}(\mathbb{R})$ and $u \in \mathbb{R}$,
	\begin{align*}
	\hat{f}(\pi_{\lambda, \mu}) h(u) &= (2\pi) \int_{\mathbb{R}} \int_{\mathbb{R}} e^{i(u-v)\xi} \mathcal{F}_{\mathbb{R}^4}(f)(\xi,\frac{\lambda}{2}v^2-\frac{\mu}{2\lambda},-\lambda v,\lambda)h(v)\, d\xi\, dv \\
	&= \int_{\mathbb{R}} \mathcal{K}_{f,\lambda,\mu}(u,v) h(v)\,dv,
	\end{align*}
	where $\mathcal{K}_{f,\lambda,\mu}(u,v)$ is the integral kernel of $\hat{f}(\pi_{\lambda, \mu})$ hence given by
	\[
	\mathcal{K}_{f,\lambda,\mu}(u,v)=2\pi \int_{\mathbb{R}}e^{i(u-v)\xi} \mathcal{F}_{\mathbb{R}^4}(f)(\xi,\frac{\lambda}{2}v^2-\frac{\mu}{2\lambda},-\lambda v,\lambda)\, d\xi.
	\]
	Using the properties of the Euclidean Fourier transform we may rewrite this as
	\[
	\mathcal{K}_{f,\lambda,\mu}(u,v)=(2\pi^{\frac{3}{2}}) \mathcal{F}_{\mathbb{R}^3}(f)(u-v,\frac{\lambda}{2}v^2-\frac{\mu}{2\lambda},-\lambda v,\lambda),
	\]
	where the Fourier transform above is taken with respect to the second, the third and the fourth variable of $f$.
	\vskip 0.2in
The $L^2(\mathbb{R}\times \mathbb{R})$-norm of the integral kernel is
\begin{align*}
\int_{\mathbb{R}} \int_{\mathbb{R}} \vert \mathcal{K}_{f,\lambda,\mu}(u,v) \vert ^2\, du\,dv
&= (2\pi^{\frac{3}{2}}) \int_{\mathbb{R}} \int_{\mathbb{R}}\vert \mathcal{F}_{\mathbb{R}^3}(f)(u-v, \frac{\lambda}{2}v^2-\frac{\mu}{2\lambda}, -\lambda v, \lambda) \vert^2 \,du\,dv\,,
\end{align*}
where since $f \in \mathcal{S}(\mathcal{B}_{4})$, the last quantity is finite. Hence the operator $\hat{f}(\pi_{\lambda, \mu})$ is Hilbert-Schmidt and its Hilbert-Schmidt norm is the exactly the above quantity. This shows the first part of the statement.
We now integrate each side of the last equality against $d\lambda, d\mu$ and obtain
\begin{eqnarray*}
\lefteqn{\int_{\mathbb{R}\setminus \{0\}} \int_{\mathbb{R}}\int_{\mathbb{R}} \int_{\mathbb{R}} \vert \mathcal{K}_{f,\lambda,\mu}(u,v) \vert ^2\, du\,dv\,d\mu\,d\lambda}\\
&=& (2\pi^{\frac{3}{2}})	\int_{\mathbb{R}\setminus \{0\}} \int_{\mathbb{R}}\int_{\mathbb{R}} \int_{\mathbb{R}} \vert  \mathcal{F}_{\mathbb{R}^3}(f)(u-v,\frac{\lambda}{2}v^2-\frac{\mu}{2\lambda},-\lambda v,\lambda)\vert ^2 du\,dv\,d\mu\,d\lambda\\
&=& (2\pi^{\frac{3}{2}}) \int_{\mathbb{R}\setminus \{0\}} \int_{\mathbb{R}}\int_{\mathbb{R}} \int_{\mathbb{R}} \vert \mathcal{F}_{\mathbb{R}^3}(f) (x_1,w_2,w_3,w_4)\vert ^2 det(J_F(u,v,\lambda, \mu)) \, dx_1\, dw_2\, dw_3\, dw_4\\
&=& \pi^{\frac{3}{2}} \cdot \int_{\mathbb{R}} \int_{\mathbb{R}\setminus\{0\}}  \int_{\mathbb{R}} \int_{\mathbb{R}} \vert  \mathcal{F}_{\mathbb{R}^3}(f) (x_1,w_2,w_3,w_4)\vert ^2 \, dw_2\, dw_3\, dw_4\, dx_1\,,\\
\end{eqnarray*}
where $det(J_F(u,v,\lambda, \mu))=\frac{1}{2}$ is the determinant of the Jacobian matrix of the linear transformation $F(u,v,\lambda, \mu)=(x_1=u-v,w_2=\frac{\lambda}{2}v^2-\frac{\mu}{2\lambda},w_3=-\lambda v,w_4=\lambda)$. Now, using the Euclidean Plancherel formula on $\mathbb{R}^3$ in the variable $(w_2,w_3,w_4)$ with dual variable $(x_2,x_3,x_4)$, we have
\[
\int_{\mathbb{R}\setminus \{0\}} \int_{\mathbb{R}}\int_{\mathbb{R} \times \mathbb{R}} \vert \mathcal{K}_{f,\lambda,\mu}(u,v) \vert ^2\, dx_1\,du\,dv\,d\mu\,d\lambda=\\
\pi^{\frac{3}{2}} \cdot\int_{\mathbb{R}^4} \vert  f (x_1,x_2,x_3,x_4)\vert ^2 \,dx_1\, dx_2\, dx_3\, dx_4. \label{plancherel3}
\]
Notice that in several parts of the proof we interchanged the order of integration as a simple application of Fubini's Theorem.
\end{proof}


\begin{thebibliography}{HOHOLT08}

\bibitem[AR20]{AR20}
R.~Akylzhanov and M.~Ruzhansky.
\newblock $L^p-L^q$ multipliers on locally compact groups.
\newblock {\em 	arXiv:1510.06321v3}, 2017.

\bibitem[AK18]{AK18}
F.~Almaki and V.~V. Kisil.
\newblock Geometric dynamics of a Harmonic oscillator, arbitary minimal uncertainly states and the smallest step 3 Nilponet {L}ie group.
\newblock{ \em arXiv:1805.01399}, 2018.

\bibitem[BFKG12]{BFKG12}
H.~Bahouri, C.~Fermanian-Kammerer and I.~Gallagher.
\newblock{\em Phase-space analysis and pseudo-differential calculus on the Heisenberg group},
\newblock Ast\'{e}risque, \textbf{342}, 2012. See also revised version of March 2013 of arXiv:0904.4746.

\bibitem[BBR96]{BBR96}
P.~Boggiatto, E.~Buzano and L.~Rodino.
\newblock {\em Global hypoellipticity ans spectral theory}, volume 92 of {\em Mathematical Research}.
\newblock Akademie Verlag, Berlin, 1996. 

\bibitem[BGR10]{BGR10}
U.~Boscain, J.~P. Gauthier and F.~Rossi.
\newblock Hypoelliptic heat kernel over 3-step nilpotent Lie groups.
\newblock{ \em J. Math. Sci.},Vol.199, No.6, 2014.

\bibitem[CDR18]{CDR18}
M.~Chatzakou, J.~Delgado and M.~Ruzhansky.
\newblock On a class of anharmonic oscillators.
\newblock{ \em arXiv:2485649}, 2018.

\bibitem[CGGP92]{CGGP92}
M.~Christ, D.~Geller, P.~Glowacki and L.~Polin.
\newblock { Pseudodiffererential operators on groups with dilations}.
\newblock {\em Duke Math. J.},\textbf{68}, 379-423, 1998.

\bibitem[CG90]{CG90}
L.~J. Corwin and F.~P. Greenleaf.
\newblock Representations on nilpotent Lie groups and their applications.
\newblock {\em Part I}, volume 18 of {\em Cambridge Studies in Advanced Mathematics}. Cambridge University Press, Cambridge, 1990. Basic theory and examples.
 
 \bibitem[Dix81]{Dix81}
 J.~Dixmier.
 \newblock {\em Von Neumann algebras}. Amsterdam ; New York : North-Holland Pub. Co, 1981. 
 
\bibitem[Dix57]{Dix57}
J.~Dixmier.
\newblock Sur les repr\'{e}sentations unitaires des groupes de Lie nilpotents, volume III of {\em Canad. J. Math.10}, pages 321-348, 1957.

\bibitem[Dix77]{Dix77}
J.~Dixmier.
\newblock {\em $C^{*}$ algebras}.
\newblock North-Holland Publishing Co., Amsterdam, 1977.
\newblock Translated from the French by Francis Jellet, North-Holland Mathematical library, Vol.15.

\bibitem[Dix81]{Dix81}
J.~Dixmier.
\newblock {\em non Neumann algebras}, volume 27 of North-Holland Mathematical library. North-Holland Publishing Co., Amsterdam, 1981.
\newblock With a preface by E.~C. Lance, Translated from the French edition by F.~Jellet.

\bibitem[Dyn76]{Dyn76}
A.~S. Dynin.
\newblock An algebra of pseudodifferential operators on the Heisenberg groups. Symbolic calculus,
\newblock {\em Dokl. Akad. Nauk SSSR}, \textbf{227}, 792-795, 1976.

\bibitem[FR14]{FR14}
V.~Fischer and M.~Ruzhansky.
\newblock A pseudo-differential calculus on graded nilpotent Lie groups.
\newblock{\em In Fourier analysis},Trends Math., pages 107-132. 
\newblock Birkh\"auser/Springer, Cham, 2014.

\bibitem[FR16]{FR16}
V.~Fischer and M.~Ruzhansky.
\newblock {\em Quantization on nilpotent {L}ie groups}, volume 314 of {\em
  Progress in Mathematics}.
\newblock Birkh\"auser/Springer, [Open access book], 2016.

\bibitem[Fol94]{Fol94}
G.B.~Folland.
\newblock Meta-Heisenberg groups, Fourier analysis (Orono, ME, 1992), Lecture notes in Pure and Appl. Math.,
\newblock \textbf{157}, 121-147, Dekker, New York, 1994.

\bibitem[FS82]{FS82}
G.B.~ Folland and E.~Stein.
\newblock {\em Hardy spaces on Homogeneous groups}.
\newblock Mathematical Notes \textbf{28}.
Princeton University Press, 1982.

\bibitem[Fol75]{Fol75}
G.B.~ Folland.
\newblock Subelliptic estimates and function spaces on nilpotent {L}ie groups.
\newblock {\em Ark. Mat.}, 13(2):161--207, 1975.

\bibitem[FS74]{FS74}
G.B.~Follanad and E.M.~Stein.
\newblock {\em Estimates for the $\overline{\partial}_{b}$ complex and analysis on the Heisenberg group.}
\newblock{Comm. Pure Appl. Math.}, 27, 1974, 429-522.

\bibitem[HK16]{HK16}
A.~Hassannezhad and G.~Kokarev.
\newblock Sub-Laplacian eigenvalue bounds on sub-Riemannian manifolds. \newblock {\em Ann.Scuola Norm. Sup Pisa Cl. Sci},XVI(4):1049-1092, 2016.

\bibitem[H\"{o}r85]{Hor85}
L.~H\"{o}rmander.
\newblock {\em The Analysis of linear partial differential operators}, vol. III,  \newblock {\em Springer-Verlang}, 1985. 

\bibitem[H\"{o}r60]{Hor60}
L.~H\"{o}rmander.
\newblock {\em  Estimates for translation invariant operators in $L^{p}$ spaces}.
\newblock Acta Math., 104:93–140, 1960.

\bibitem[RS77]{RS77}
L.P.~Rothscild and E.M.~ Stein.
\newblock{ \em Hypoelliptic differential operators and nilpotent groups.}
\newblock { Acta Math.}, 137, 1976, 3-4, 247-320.


\bibitem[RR18]{RR18}
D.~Rottensteiner and M.~Ruzhansky.
\newblock {\em Harmonic and Anharmonic oscillators on the Heisenberg group.}
\newblock {\em arXiv:1812.09620v1}, 2018.

\bibitem[RT10]{RT10}
M.~Ruzhansky and V.~Turuner.
\newblock{\em Pseudo-differential operators and symmetries: Background analysis and advanced topics.}
\newblock{Pseudo-Differential Operators: Theory and Applications},\textbf{2}, Birkh\"{a}user, Verlag, 2010.


\bibitem[Shu01]{Shu01}
M.A.~ Shubin.
\newblock {\em Pseudodifferential operators and spectral theory. }
\newblock Springer-Verlag, Berlin, second edition, 2001. 
\newblock Translated from the 1978 Russian original by Stig I. Anderson.

\bibitem[Tayl84]{Tayl84}
M.E.~ Taylor.
\newblock {\em Noncommutative microlocal analysis. I. Mem.}
\newblock Amer. Math. Soc. \textbf{52}, 1984.
 
\bibitem[Tit58]{Tit58}
E.C.~ Titchmarsh 
\newblock {\em Eigenfunction expansions associated with second-order differential equations}. Oxford University Press, Oxford, 1958.



\end{thebibliography}
 \end{document}